\documentclass[11pt]{article}
\usepackage[pagewise]{lineno} 
\usepackage{amsmath,amsfonts,amssymb,amsthm}
\usepackage{mathtools,mathrsfs,yhmath}
\usepackage{graphicx,color,xcolor}
\usepackage{enumitem}
\usepackage{hyperref}
\usepackage{blindtext}
\usepackage{multicol}
\usepackage{color}
\usepackage{comment}
\usepackage{wrapfig}
\usepackage{dsfont}
\usepackage{graphicx}
\hypersetup{colorlinks=true, linkcolor=blue, filecolor=magenta, urlcolor=blue}
\usepackage[capitalize, nameinlink]{cleveref}
\usepackage{thmtools}
\usepackage[backend=biber, sorting=nty, doi=false, url=false, isbn=false, maxbibnames=5]{biblatex}
\bibliography{MultispeciesLandau}
\usepackage{hyperref}

\newcommand\blfootnote[1]{%
	\begingroup
	\renewcommand\thefootnote{}\footnote{#1}%
	\addtocounter{footnote}{-1}%
	\endgroup
}

 \allowdisplaybreaks
\numberwithin{equation}{section}
\crefname{enumi}{}{parts}
\numberwithin{equation}{section}
\newtheorem{theorem}{Theorem}[section]
\newtheorem{lemma}[theorem]{Lemma}
\newtheorem{proposition}[theorem]{Proposition}
\newtheorem{corollary}[theorem]{Corollary}

\newtheorem{definition}[theorem]{Definition}

\crefname{assumption}{Assumption}{Assumptions}

\newtheorem{remark}[theorem]{Remark}
\newcommand\numberthis{\addtocounter{equation}{1}\tag{\theequation}}

\newcommand{\ud}{\,\mathrm{d}}
\def\d{\,\mathrm{d}}
\newcommand{\1}{\mathds{1}}
\newcommand{\Id}{\operatorname{Id}}
\newcommand{\R}{\mathbb{R}}
\newcommand{\T}{\mathbb{T}}
\renewcommand{\S}{\mathbb{S}}
\newcommand{\lp}{\left(}
\newcommand{\rp}{\right)}

\let\mc=\mathcal

\DeclarePairedDelimiter{\abs}{\lvert}{\rvert}
\DeclarePairedDelimiter{\norm}{\lVert}{\rVert}
\makeatletter
\let\oldabs\abs
\def\abs{\@ifstar{\oldabs}{\oldabs*}}
\let\oldnorm\norm
\def\norm{\@ifstar{\oldnorm}{\oldnorm*}}
\makeatother

\newcommand{\Feq}{F_{\mathrm{eq}}}
\newcommand{\MTm}{\mathcal{M}_{T,m_i}}

\newcommand{\snabla}{\nabla_{\mathbb{S}^2}}
\newcommand{\hznabla}{\nabla_{\hz}}
\newcommand{\rpartial}{\partial_{|z|}}

\newcommand{\mi}{m_i}
\newcommand{\mj}{m_j}

\newcommand{\Mij}{M_{ij}} 
\newcommand{\aij}{\alpha_{ij}}

\newcommand{\gradvi}{\nabla_{v_i}}
\newcommand{\gradvj}{\nabla_{v_j}}
\newcommand{\gradz}{\nabla_z}
\newcommand{\gradzs}{\nabla_{z^*}}

\newcommand{\nablaijminus}{\left(\frac{\gradvi}{\mi}-\frac{\gradvj}{\mj}\right)}
\newcommand{\nablaijplus}{\left(\gradvi + \gradvj\right)}
\newcommand{\nablavwminus}{\left(\nabla_v - \nabla_{v^*}\right)}

\newcommand{\Pperp}[1]{\Pi^{\perp}_{#1}}         

\newcommand{\Aij}{A(v_i-v_j)}              
\newcommand{\cAij}{c_{ij}A(v_i-v_j)}              
\newcommand{\Az}{A(z)}

\newcommand{\fij}{\left(f_i \otimes f_j\right)}
\newcommand{\Uij}{\log \fij}
\newcommand{\ofij}{\smash[b]{\underline{f}}_{ij}}
\newcommand{\ufij}{\overline{f}_{ij}}
\newcommand{\ff}{\left(f \otimes f\right)}
\newcommand{\off}{\smash[b]{\underline{f}}_{\otimes}}
\newcommand{\uff}{\overline{f}_{\otimes}}

\newcommand{\tF}{\widetilde{F}}
\newcommand{\tQ}{\widetilde{Q}}
\newcommand{\tI}{\widetilde{\mathcal{I}}}
\newcommand{\tIij}{\widetilde{\mathcal{I}}_{ij}}

\newcommand{\tf}{\overline{f}}
\newcommand{\tests}{\overline{g}}
\newcommand{\testf}{g}

\newcommand{\marg}[1]{\Pi_{#1}\tF}

\newcommand{\hz}{\hat{z}}
\newcommand{\hv}{\hat{v}}

\newcommand{\bbnabla}{%
    \,\nabla\mkern-12mu\nabla_{i,j}%
}
\newcommand{\bbnablavw}{%
    \,\nabla\mkern-12mu\nabla%
}
\newcommand{\notbbnabla}{%
  \mathrel{\ooalign{%
    $\nabla\mkern-12mu\nabla$\cr
    \hidewidth\raise0.2ex\hbox{$/$}\hidewidth}}_{i,j}\mkern-2.5mu%
}
\newcommand{\notbbnablavw}{%
  \mathrel{\ooalign{%
    $\nabla\mkern-12mu\nabla$\cr
    \hidewidth\raise0.2ex\hbox{$/$}\hidewidth}}%
}
\newcommand{\notsbbnabla}{%
  \mathrel{\ooalign{%
    $\nabla\mkern-12mu\nabla$\cr
    \hidewidth\raise0.2ex\hbox{$/$}\hidewidth}}_{i,j}^{\;*}%
}

\newcommand{\IFish}[1]{I\!\left(#1\right)}          
\newcommand{\IFishMulti}[1]{\mathcal{I}\!\left(#1\right)}     
\newcommand{\HRel}[2]{H\!\left(#1|#2\right)}                     
\newcommand{\JSph}[1]{J\!\left(#1\right)}                     

\hypersetup{pdftitle={MultispeciesLandau}}
\hypersetup{pdfauthor={Jonathan Junn\'e, Raphael Winter  \& Havva Yoldaş }}

\author{Jonathan Junn\'e\footnote{Delft Institute of Applied Mathematics, Faculty of Electrical Engineering, Mathematics and Computer Science, Delft University of Technology, Mekelweg 4, 2628CD Delft, The Netherlands. \href{mailto:j.junne@tudelft.nl}{J.Junne@tudelft.nl} \& \href{mailto:H.Yoldas@tudelft.nl}{H.Yoldas@tudelft.nl}}
\and Raphael Winter\footnote{School of Mathematics, Cardiff University,  Abacws, Senghennydd Road, Cathays, Cardiff CF24 4AG, United Kingdom.  \href{mailto:WinterR6@cardiff.ac.uk}{WinterR6@cardiff.ac.uk}}
\and Havva Yolda\c{s}$^{\ast}$}

\title{On the existence of solutions to the multi-species Landau equation}

\begin{document}
\maketitle

\begin{abstract}
 We consider the spatially homogeneous Landau equation for multiple species with different masses. As in the single-species case, the singularity of the collision operator is determined by a parameter $\gamma \in [-3,1]$, where $\gamma = -3$ corresponds to Coulomb interactions. We prove that if $\gamma\geq -\sqrt{8}$ in the cross-interaction operators, then there exists a natural multi-species generalization of the Fisher information which is a Lyapunov functional for the multi-species Landau system. On the other hand, we give a counterexample showing that the Fisher information is in general no longer a Lyapunov functional below the threshold $(\gamma < - \sqrt{8})$ for the two-species system if one species has infinite mass. However, we are able to provide a new method to show global well-posedness, by constructing a different Lyapunov functional based on the spherical Fisher information. 

\end{abstract}

\tableofcontents

\section{Introduction}

We study the spatially homogeneous multi-species Landau equation, a kinetic model for the time evolution of a plasma consisting of several distinct particle species, such as ions and electrons. The homogeneous multi-species Landau system, satisfied by the probability densities $f_i=f_i(t,v)$ with $(t,v) \in (0,\infty)\times\R^3$ for $N$ species, is given by
\begin{equation} \label{eq:multi_Landau}
	\partial_t f_i = \sum_{j=1}^N Q_{ij}(f_i,f_j), 
	\qquad 1\leq i\leq N,
\end{equation}
where $Q_{ij}(f_i,f_j)$ is the Landau collision operator, derived by Landau in 1936 \cite{L58},
\begin{align*}
	Q_{ij}(f_i,f_j)(v_i)
  &= \frac{1}{\mi}\,\gradvi \cdot \left(
      \int_{\R^3} c_{ij} A(v_i-v_j)\,
        \left( \frac{\gradvi}{\mi}-\frac{\gradvj}{\mj} \right)
        \fij(v_i,v_j)\,\ud v_j
    \right),
\end{align*}
with
\begin{align} \label{eq:APiC}
    A(z) \coloneqq \, \abs{z}^{2+\gamma}\Pperp{z}, \quad \Pperp{z} \coloneqq \Id - \frac{z\otimes z}{\abs{z}^2}, \quad c_{ij} \coloneqq \frac{\abs{\log{\Lambda_c}} q_i^2 q_j^2 n_i n_j}{8 \pi \varepsilon_0^2},
\end{align}
where $|\log \Lambda_c|$ is the Coulomb logarithm, and $\varepsilon_0$ is the vacuum permittivity. The constants $n_i$ are the number densities of each species per unit volume, so we can assume $f_i$ to be probability densities. The constant $\gamma \in [-3,1]$ determines the singularity of the interaction kernel, the physically most relevant case describing the interaction of charged particles via Coulomb potentials corresponds to $\gamma=-3$. The Landau equation can be obtained in the \emph{grazing collisions limit} from the Boltzmann equation, see e.g.,~\cite{AV04, De92, Vi98, BW23}.

The breakthrough result of Guillen and Silvestre~\cite{GS25} establishes the global existence of solutions for the Landau-Coulomb equation, when $f$ solves \eqref{eq:multi_Landau} for only one species. They achieve this by showing that the Fisher information
\begin{align}
    \IFish{f} = \int_{\R^3} \abs{\nabla \log f}^2 f
\end{align}
is non-increasing along the solution of the Landau equation. On the other hand, for the multi-species Landau equation only a couple of mathematically rigorous results are available. 

As the multi-species Landau equation arises from (grazing) elastic collisions it conserves, at least formally, the number of particles for all species, as well as the total momentum $M$ and total energy $E$. Therefore, we have the following conserved quantities
\begin{align} \label{eq:total_momentum_energy}
    M &:= \sum_{i=1}^N \int_{\R^3} \mi v\, f_i(v) \d v, \quad E := \frac12 \sum_{i=1}^N \int_{\R^3} \mi \abs{v}^{2} f_i(v) \d v.
\end{align}
Moreover, the Boltzmann entropy $\mc H $, given by
\begin{align}
    \mathcal{H}(f_1,\ldots, f_N) &:= \sum_{i=1}^N \int_{\R^3} f_i \log f_i,
\end{align} 
is formally non-increasing in time. 

In this paper, we tackle the problem of global well-posedness for the spatially homogeneous multi-species Landau equation. To this end, we first derive a natural generalization of the Fisher information in the multi-species setting. Since the global equilibrium of the multi-species system is given by the Gibbs state 
\begin{align} \label{Gibbs}
     \Feq (v_1, \ldots, v_N) = \prod_{i=1}^N \MTm(v_i),
\end{align}
where $T>0$ is the temperature and $\MTm$ is the Maxwellian velocity distribution for particles of mass $m_i$, i.e. 
\begin{align} \label{def:Maxwellian}
    \MTm(v) = \left( \frac{m_i}{2\pi T}\right)^{3/2} 
              \exp\left(-\frac{m_i}{2T}\abs{v}^2\right),
\end{align}
it is natural to seek a multi-species Fisher information $\mathcal I$ of the form
\begin{align*}
    \IFishMulti{F_N;{\bf a}} 
    = \sum_{i=1}^N a_i \int_{\R^3} \abs{\nabla \log f_i}^2 f_i,  
\end{align*}
where $F_N = \lp f_i\rp_{i=1}^N$ solution to \eqref{eq:multi_Landau}, and ${\bf a} =(a_1,\dots, a_N)$ positive constants $a_i>0$ to be chosen. We can only expect monotonicity of this functional if the equilibrium~\eqref{Gibbs} is a minimizer among all configurations with fixed total kinetic energy $E$ defined in~\eqref{eq:total_momentum_energy}. Testing this condition on factorized functions with the same kinetic energy shows that, up to a multiplicative constant, the unique choice is $a_i=\mi^{-1}$. We therefore define
\begin{align} \label{def:nFisher}
    \IFishMulti{F_N} 
    \coloneqq \sum_{i=1}^N \frac{1}{\mi} \int_{\R^3} \abs{\nabla \log f_i}^2 f_i.
\end{align}

This functional is indeed a Lyapunov functional for the homogeneous multi-species Landau equation provided that $\gamma \in [-\sqrt{8}, -2]$ which is our first main result \cref{thm:Global}. We recast the cross-interaction operators in a new set of variables~\eqref{eq:new_variables} and differential operators~\eqref{gradnotation}. After this transformation, the operators can be treated with similar to~\cite{GS25}, and the constant in the log-Sobolev inequality without additional symmetry is responsible for the threshold value $-\sqrt{8}$. We can even prove that for $\gamma \in [-3, -\sqrt{8})$ the Fisher information $\mc I (F_N)$ is in general no longer a Lyapunov functional, by giving an explicit counterexample in~\cref{thm:FisherIncrease}. 

The first two theorems pose the question whether the multi-species Landau system is indeed globally well-posed below the threshold value  $-\sqrt{8}$. This is the main objective of our analysis. We are able to answer this question positively in the case of two-species Landau system where one species has infinite mass, i.e. the collisional dynamics of electrons and ions. While our counterexample shows that the time derivative of the Fisher information along this dynamics is no longer under control, we construct a new Lyapunov functional $\Lambda (f)$, defined in~\eqref{eq:new_Lyapunov}, by combining a weighted spherical Fisher information, the classical Fisher information and the Boltzmann entropy. With this new functional we establish global well-posedness for $\gamma \in [-3, -\sqrt{8})$ in~\cref{thm:Lambda dissipation}.

\subsection{Summary of previous results} \label{sec:summary}

In this section, we summarize the theory of the Landau equation with a particular focus on the results after~\cite{GS25}. We refer the reader to~\cite{GS25} for a detailed review and historical remarks. The global well-posedness theory for the spatially homogeneous Landau equation is well-established for hard potentials and Maxwellian molecules, i.e., $\gamma \geq 0$, see e.g., \cite{Vi98-2} for $\gamma =0$ and \cite{DV00, DV00-2} for  $\gamma >0$. For moderately soft potentials, $\gamma \in (-2, 0]$, regularity estimates developed in \cite{GG18, Si17, Wu14} imply global smooth solutions. For the very soft potentials, i.e., $\gamma \in [-3,-2)$ including the case of Coulomb interactions, this problem was open until the recent breakthrough result by Guillen and Silveste \cite{GS25}. They prove that the Fisher information is monotone decreasing for the spatially homogeneous Landau equation with Coulomb interactions, which gives the existence of global smooth solutions. Since then, this new information-theoretical approach has been used to address other open problems in kinetic theory. We refer the reader to the recent notes by Villani \cite{Vi25} on this topic. In \cite{ISV25}, Imbert, Silvestre and Villani generalized this method to show the existence of global smooth solutions to the space-homogeneous Boltzmann equation in the regime of very soft potentials. The Fisher information approach has been generalized to initial data without finite Fisher information, see~\cite{CGG25, DGGL24, GGL25, Ji24-2}. Let us briefly discuss the Lenard-Balescu equation which is considered as a more accurate model than Landau equation for weakly coupled systems with long-range forces (especially Coulomb plasmas). The presence of collective effects leads to an additional non-local nonlinearity in the collision operator and currently no Fisher information based theory is known. The state of the art regarding the spatially homogeneous Lenard-Balescu equation is due to~\cite{BD25, DW23, St07}. 

The spatially in-homogenous Landau equation with Coulomb interactions still remains open for general initial data. However, it is possible to obtain rigorous results when the collision operator is mollified in space. In~\cite{GGPTZ25}, the authors study a fuzzy version of the spatially in-homogeneous Landau equation where the collisions are delocalized via a spatially dependent kernel $\kappa (x-x_*)$. They prove the global well-posedness of the fuzzy Landau equation for moderately soft interactions where $\gamma \in (-2,0]$ and a positive spatial kernel $\kappa (x) \sim \langle x\rangle^{-\lambda},  \lambda >0 $ when the initial data is sufficiently smooth. For $\kappa \equiv 1$, and $\gamma \in [-3,1]$, they also prove that the spatial Fisher information decays monotonically and the full Fisher information remains uniformly bounded in finite time intervals. We also refer to recent article series \cite{DH25-3, DH25, DH25-2} for more details on the fuzzy Landau equation.  

When it comes to the multi-species Landau equation, mathematical results are rather scarce. Classical references for the derivation of \eqref{eq:multi_Landau} are \cite[Chapter 4]{LP83} and \cite[Chapter 6]{Sch91}. This problem is closely related to the study of multi-species Boltzmann equation; thus, the existing results use similar approaches. Let us start by summarizing some landmark results in this case. In \cite{SY10}, the authors study the Boltzmann equation for a mixture of two gases in one-spatial dimension, and prove sharp pointwise-in-$(t,x)$ asymptotics where one species is near vacuum and the other one is near Maxwellian equilibrium. One of the first spectral gap results for the spatially in-homogeneous multi-species Boltzmann system on $(x,v) \in \T^3 \times \R^3$ covering hard potentials and Maxwellian molecules, i.e., $\gamma  \in [0,1]$ with Grad's cutoff assumption is \cite{DJMZ16}. This work uses a multi-species adaptation of the quantitative hypocoercivity techniques developed in \cite{MN06}. However, \cite{DJMZ16} does not allow species to have different masses in the mixture. In \cite{BD16}, the authors develop a perturbative Cauchy theory for the multi-species Boltzmann equation posed on $(x,v) \in \T^3 \times \R^3$ in the weighted space $L_v^1L_x^\infty $ with a polynomial weight. They prove the existence of a spectral gap for the linear multi-species Boltzmann operator allowing different masses which induce a loss of symmetry and the standard methods developed for the mono-species case are not immediately available. They study the perturbed linear equation by means of $L^1-L^\infty$ theory due to \cite{Gu10}. 

To the best of our knowledge, there are only two works on the Cauchy problem of the multi-species Landau system. In \cite{GZ17}, the authors study the system of the spatially in-homogeneous Landau system defined in the phase space $(x,p) \in \T^3 \times \R^3 $ with moderately soft potentials, i.e., $\gamma \in [-2,1]$. Using a similar strategy as in \cite{DJMZ16}, they provide explicit spectral gap and hypocoercivity estimates for a linearized multi-species Landau system where all species are assumed to be close to equilibrium. In \cite{LWW18}, the authors study the Landau system on $(x,p) \in \R^3 \times \R^3$ for $\gamma \in [-2,1]$ for $2$ species where one species starts near vacuum and the other one starts near a Maxwellian equilibrium state. Using a similar approach in the case of the multi-species Boltzmann system \cite{SY10}, they show that the solutions to the linearized system become instantaneously smooth both in $x$ and $p$ without requiring any smoothness assumption on the initial datum. On the numerics side, \cite{CHF24} develops a deterministic, structure-preserving particle method for the spatially homogeneous, multi-species Landau equation by regularizing the collision operator so that each species distribution can be approximated as a sum of Dirac masses whose locations evolve according to an ODE system. 

\subsection{Main results} \label{sec:main-results}

\begin{theorem}[Global well-posedness] \label{thm:Global}
    Let $N \in \mathbb N$ be the number of species and  $m_i>0$, $1\leq i\leq N$ their respective particle masses. Assume that the interaction is soft with $-2 >\gamma\geq -\sqrt{8}$. Further assume that the initial distributions satisfy  $f^\circ_i \in  L^1_{2\ell} \cap L^2_\ell \cap L \log L$ for some  $\ell\geq 7$ and are probability densities.   Then there exists a global-in-time strong solution $F_N=(f_i)_{i=1}^N$ which solves \eqref{eq:multi_Landau}, and
    \begin{align*}
        f_i \in C^\infty((0,\infty)\times \R^3), \quad \text{for } 1\leq i\leq N.        
    \end{align*}
    Moreover, for any $t>0$, the generalized Fisher information $\IFishMulti{F_N}$ defined in~\eqref{def:nFisher} is finite and non-increasing in time on $(0,\infty)$. 
\end{theorem}
The proof of this theorem is the content of~\cref{subsec:Thm1Proof}.

The theorem covers a large range of soft potentials, down to the threshold $\gamma_* = -\sqrt{8}$. This naturally raises the question whether the monotonicity of the Fisher information persists for $\gamma<-\sqrt{8}$. We answer this question negatively in the limiting case where one species has infinite mass. This equation models, in particular, the collisional dynamics of electrons interacting with ions, since the mass ratio is about 1836 already between protons and electrons. The resulting equation can be found in~\cite[Equation 6.4.11]{Sch91} and reads 
\begin{align} \label{eq:InfiniteGamma}
    \partial_t f = Q_\gamma(f,f) + c_{ei} Lf,
\end{align}
with the choice $\gamma= -3$ and $c_{ei}>0$ in the general formula
\begin{align*}
    Q_\gamma(f,f)(v) &\coloneqq \nabla_v \cdot \int_{\R^3} A(v-v^*) \nablavwminus \ff (v,v^*) \ud{v^*}, \\ 
    Lf(v) &\coloneqq \nabla_v \cdot \left(|v|^{2+\gamma} \Pperp{v} \nabla_v f(v)\right).
\end{align*}
In other words, the interaction with ions induces a diffusion on spheres in velocity space. We prove that the Fisher information is a Lyapunov function for this equation if the initial datum $f^\circ$ is \emph{even}, but is \emph{not} monotone decreasing in general.

\begin{theorem}[Non-monotonicity of the Fisher information below the critical threshold] \label{thm:FisherIncrease}
    Let $-3 \leq \gamma <- \sqrt{8}$. Then there exists a probability density $f^\circ \in \mathcal{S}(\R^3)$ such that the local-in-time solution $f$ with initial data $f^\circ$ to~\eqref{eq:InfiniteGamma} satisfies
    \begin{align*}
        \frac{\d}{\d t} \IFish{f}\big|_{t=0} > 0.
    \end{align*}

    On the other hand, if $f^\circ$ is \emph{even}, i.e. $f^\circ (v)=f^\circ (-v)$, then the Fisher information $I(f)$ is a Lyapunov functional for~\eqref{eq:InfiniteGamma} for any $-3\leq \gamma \leq -2$. 
\end{theorem}
The proof of this theorem is the content of~\cref{sec:nonMono}.

This theorem shows that the Fisher information is no longer a Lyapunov functional. Moreover, the proof reveals that the time derivative of the Fisher information cannot be controlled in terms of $\IFish{f}$ alone, which makes any direct growth estimate for $\IFish{f}$ inaccessible. This prompts the question whether~\eqref{eq:InfiniteGamma} is in fact globally well-posed. We answer this positively even for Coulomb interactions between electrons and ions. We consider the equation 
\begin{align} \label{eq:infiniteCoulomb}
    \partial_t f = Q_\gamma(f,f)+ c_{ei} \nabla \cdot \lp \frac{\Pperp{v}}{\abs{v}}\, \nabla f \rp,
\end{align}
where $Q_\gamma(f,f)$ is given by the single-species Landau operator with coefficient $\gamma\in [-3,2]$ in~\eqref{eq:APiC}.

To achieve this, we introduce, in addition to the standard Fisher information, a weighted spherical Fisher information $\JSph{f}$ defined by
\begin{align*}
    \JSph{f} \coloneqq \int_{\R^3} \frac{\abs{\Pperp{v} \nabla \log f(v)}^2}{\abs{v}}\, f(v)\,\ud v,
\end{align*}
together with the Boltzmann relative entropy $\HRel{f}{\mathcal{M}_{{T^\circ}, 1}}$ with respect to the Maxwellian $\mathcal{M}_{{T^\circ}, 1}$ \eqref{def:Maxwellian} with temperature $T^\circ$ corresponding to the conserved free energy, given by
\begin{align*}
    \HRel{f}{\mathcal{M}_{{T^\circ}, 1}} \coloneqq H(f) - \int_{\R^3} f \log \mathcal{M}_{{T^\circ}, 1} \coloneqq \int_{\R^3} f \log f - \int_{\R^3} f \log \mathcal{M}_{{T^\circ}, 1}.
\end{align*}
Here, $T^\circ$ is chosen such that
\begin{align*}
    \int_{\R^3} \frac12 |v|^2\mathcal{M}_{{T^\circ}, 1}(v)\, \d v  = \int_{\R^3} \frac12 |v|^2 f^\circ(v)\, \d v = E.
\end{align*}
We are now in the position to state the main result.

\begin{theorem}[New Lyapunov functional for the infinite-mass case] \label{thm:Lambda dissipation}
Let $-3 \leq   \gamma < -\sqrt{8}$ and $c_{ei}>0$ and  $f$ be the (local) strong solution to~\eqref{eq:infiniteCoulomb} with initial datum given by a probability density $f^\circ$ satisfying $f^\circ \in L^1_{2\ell} \cap L^2_\ell \cap L \log L$ for some $\ell\geq 7$ and  $\JSph{f^\circ } < \infty$.
In the case of $\gamma = -3$, further assume that $c_{ei}>0$ and the initial datum $f^\circ$ satisfy
    \begin{align} \label{eq:Coulombcondition}
        \lp 2 + \frac{J(f^\circ)}{I^\frac32 (f^\circ)} +\HRel{f^\circ }{\mathcal{M}_{{T^\circ}, 1}} \rp \leq \frac{c_{ei}}{K},
    \end{align}
    where $K>0$ is a sufficiently large constant. 
Then there exist constants $a,R > 0$ such that the non-negative functional
    \begin{align} \label{eq:new_Lyapunov}
        \Lambda(f) = \IFish{f} + a\,\JSph{f} + R\,\HRel{f}{\mathcal{M}_{{T^\circ}, 1}}
    \end{align}
    is non-increasing in time. In particular, the Fisher information $\IFish{f}$ remains globally bounded by a constant depending only on the initial datum $f^\circ $, and the solution exists globally-in-time.
\end{theorem}

The proof of this theorem is the content of~\cref{sec:Thm3proof}.

Let us emphasize that this result shows that the global-in-time boundedness of the Fisher information persists, even though the Fisher information itself is no longer monotone for $\gamma < -\sqrt{8}$. 

For general finite mass-ratio of two species, the question of well-posedness of the two-species system remains open when $\gamma < - \sqrt{8}$. In the linearized setting and large mass ratio, the problem is related to the two-scale problem considered  in~\cite{GW25}.   

\subsection{Preliminaries and structure of the paper}

Throughout the paper, we follow the notational convention introduced in~\cite{GS25} and write $\langle I'(f) ,g\rangle $ for the Gateaux derivative of a functional $I$ in direction $g$.
Moreover, we recall the following standard notation for weighted spaces.
\begin{definition}[Weighted spaces]
We work in weighted spaces $L^2_l$ and $H^k_l$, $l\geq 0$, $k\in \mathbb N$, defined by
\begin{align*}
    \| f\|^2_{L^2_l}&= \int_{\R^3} f^2(v) \langle v\rangle^{l} \d v, \\
    \|f\|^2_{H^k_l} &= \sum_{k=0}^n \| \nabla^kf\|^2_{L^2_l}.
\end{align*}
To control the decay of the solution, we also use the moments
\begin{align*}
    \| f\|_{L^1_l} = \int_{\R^3} f(v) \langle v\rangle^l \d v. 
\end{align*}
\end{definition}

We introduce the number densities $n_i$ in~\eqref{eq:APiC} so that we can assume throughout the paper that $f$, $f_i$ are probability densities. 

\medskip 

We use the standard notation for the Frobenius product of real matrices and the induced Hilbert-Schmidt norm
\begin{align*}
    A : B = \sum_{i,j} A_{ij} B_{ij} = \operatorname{trace(A^t B)} \qquad \text{and} \qquad \| A\|^2_{\mathrm{HS}}= A: A.  
\end{align*}

Let us briefly summarize the structure of the paper. 
In~\cref{sec:finite-mass} we treat the finite-mass case and prove~\cref{thm:Global}. In~\cref{Sec:Infinite} we consider the case of two species, of which one has infinite particle mass. \cref{subsec:PrepInfinite} contains preliminary results on the dissipation functionals for the infinite mass case, while the proof of~\cref{thm:FisherIncrease} is carried out in~\cref{sec:nonMono}. The novel dissipation functional and the proof of~\cref{thm:Lambda dissipation} are the content of~\cref{sec:Thm3proof}. Some frequently used identities in spherical calculus are gathered in Appendix \ref{sec:appendixA}.

\section{The multi-species Landau system with finite masses}\label{sec:finite-mass}

\subsection{Lifting argument for the multi-species system}

We perform a lifting procedure by doubling the number of variables as in~\cite{GS25}. To this end, we define
\begin{equation*}
    m_{N+i} \coloneqq \mi, 
    \qquad 1 \le i \le N,
\end{equation*}
so that particles $i$ and $N+i$ belong to the same species. We introduce the lifted linear operators acting on functions
\[
  \tF = \tF(v_1, \ldots, v_{2N}) : (\R^3)^{2N} \to \R
\]
by
\begin{equation*}
    \tQ_{ij}(\tF)
    \coloneqq \nablaijminus \cdot \left( \cAij \,\nablaijminus \tF \right),
\end{equation*}
and consider the linear evolution of the Cauchy problem 
\begin{equation}\label{eq:lifted multi_Landau}
    \begin{cases}
        \partial_t \tF = \tQ(\tF)  = \displaystyle\sum_{\substack{1 \le i \le N\\ N+1 \le j \le 2N}} \tQ_{ij}(\tF), \\[0.5em]
        \tF^\circ(v_1, \dots, v_{2N}) 
          = \displaystyle\prod_{i=1}^N f^\circ_i(v_i)\, f^\circ_i(v_{N+i}). 
    \end{cases}
\end{equation}

The evolution~\eqref{eq:lifted multi_Landau} is useful because the equations for the marginals can be related back to the multi-species system~\eqref{eq:multi_Landau}. To see this, we consider the marginals
\begin{equation*}
    \marg{k} 
    = \int_{(\R^3)^{2N-1}} \tF \,\ud\widehat{\mathbf{v}}_k, 
    \qquad 1 \le k \le 2N,
\end{equation*}
where $\widehat{\mathbf{v}}_k$ denotes the vector $(v_1, \ldots, v_{2N})$ with $v_k$ omitted.
For $1 \le i \le N$, we have
\begin{equation*}
    \partial_t \marg{i}
    = \sum_{j=N+1}^{2N}\frac{\gradvi}{\mi} \cdot \left(
        \int_{(\R^3)^{2N-1}} \cAij \,\nablaijminus \tF(v_1, \ldots, v_{2N})
        \,\ud{\widehat{\mathbf{v}}_i}
    \right),
\end{equation*}
and thus, at time $t=0$, we recover the multi-species equation~\eqref{eq:multi_Landau};
\begin{equation*}
    \partial_t \marg{i}\Big|_{t = 0} 
    = \sum_{j=1}^N Q_{ij}(f_i^\circ, f_j^\circ).
\end{equation*}

For a single density, we define the Fisher information by
\begin{equation*}
    \IFish{f}
    \coloneqq \int_{\R^3} \abs{\nabla \log f}^2 f \d v.
\end{equation*}
For the lifted multi-species system we introduce the weighted Fisher information
\begin{equation*}
    \tI (\tF)
    \coloneqq \int_{(\R^3)^{2N}} 
      \sum_{i=1}^{2N} \frac{1}{\mi} 
      \abs{\nabla_{v_i} \log \tF}^2 \tF \d \mathbf{v}_{2N},
\end{equation*} where $\mathbf{v}_{2N} = (v_1, \ldots, v_{2N})$. Although this definition is asymmetric, the weighted Fisher information retains the key property that we need: it controls a suitably scaled sum of the marginals' Fisher information.

\begin{lemma} \label{lem:Fisher Carlen}
    Let $\tF:(\R^3)^{2N} \rightarrow \R$ be a probability density with finite
    weighted Fisher information $\tI(\tF)$ and marginals
    \begin{equation*}
        \marg{i} = \marg{N+i} = f_i, \qquad 1 \le i \le N.
    \end{equation*}
    Then
    \begin{equation} \label{Carlen_ineq}
        \frac12 \tI(\tF) \;\ge\; \sum_{i=1}^N \frac{1}{\mi}\,\IFish{f_i},
    \end{equation}
    with equality if and only if $\tF = \left(f_1 \otimes \cdots \otimes f_N\right)^{\otimes 2}$.
\end{lemma}

\begin{proof}
    Consider the change of variables $\Phi:(\R^3)^{2N}\to(\R^3)^{2N}$ defined by
    \begin{equation*}
      \Phi:\ (v_1,\dots,v_{2N}) \mapsto \left(\sqrt{m_1}\,v_1,\dots,\sqrt{m_{2N}}\,v_{2N}\right).
    \end{equation*}
    Then the push-forward density $G = |\det D\Phi|^{-1}\,\tF\circ\Phi^{-1}$ satisfies
    \begin{equation*}
        \int_{\R^{2N}} \abs{\nabla \log G}^2 G \d \mathbf{v}_{2N} = \tI(\tF).
    \end{equation*}
    The claim is then a direct consequence of Carlen's inequality, i.e.~\cite[Theorem~3]{Ca91}. The factor $\frac12$ stems from the fact that $f_i=f_{N+i}$, so each probability distribution appears twice as a marginal.
\end{proof}

\subsection{Dissipation of the weighted Fisher information}\label{sec:Dissipation of the weighted Fisher information}
This section is devoted to the proof of the dissipation of the weighted Fisher information, given local well-posedness. This is the fundamental tool to prove~\cref{thm:Global}. One key ingredient is to introduce the following differential operators,
\begin{equation}\label{gradnotation}
    \bbnabla \coloneqq \left(\frac{\gradvi}{\sqrt{\mi}}, \frac{\gradvj}{\sqrt{\mj}}\right),
    \qquad 
    \notbbnabla\:\: \coloneqq \frac{\gradvi}{\mi} - \frac{\gradvj}{\mj}.
\end{equation}

\begin{lemma} \label{lem:linearity}
    Let $\tF(v_1, \cdots, v_{2N}) = \otimes_{k=1}^{2N} f_k $ be a probability density with marginals $f_k$ and finite weighted Fisher information $\tI(\tF)$. Then for any pair $(i, j)\, :\, 1 \le i\le N, \, N+1\le j\le 2N$,
    \begin{equation*} 
        \big\langle \tI'(\tF), \tQ_{ij}(\tF) \big\rangle
        = \big\langle \tIij'\fij, \tQ_{ij}\fij \big\rangle,
    \end{equation*}
    where
    \begin{equation*}
    \tQ_{ij}(f_i \otimes f_j) \coloneqq\: \notbbnabla \cdot \left( \cAij\, \notbbnabla \fij \right)
    \end{equation*}
    and
    \begin{equation*}
        \tIij\fij 
        \coloneqq \int_{(\R^3)^{2}} \abs{\bbnabla \log \fij}^2 \fij.
    \end{equation*}
    As a consequence, we have the comparison principle
    \begin{align} \label{eq:DissInequality}
         \frac{\d}{\d t} \mathcal{I}(F)\big|_{t=0} \leq \frac12  \frac{\d}{\d t} \tI'(\tF) \big|_{t=0}.
    \end{align}
\end{lemma}

\begin{proof}
    By linearity, it suffices to show that for any index $k\neq i,j$ and
    \begin{equation*}
        \tI_k (\tF) \coloneqq \int_{(\R^3)^{2N}} \frac{1}{m_k}\abs{\nabla_{v_k} \log \tF}^2 \tF,
    \end{equation*}
    we have 
    \begin{equation*}
        \big\langle \tI_k' (\tF), \tQ_{ij} (\tF) \big\rangle = 0.
    \end{equation*}
    A direct computation gives
    \begin{align*}
        \big\langle \tI_k' (\tF), \tQ_{ij}(\tF) \big\rangle 
        &= \int_{(\R^3)^{2N}} \frac{2}{m_k} 
           \left(\nabla_{v_k } \left ( \frac{\tQ_{ij}(\tF)}{\tF}\right )\cdot 
                 \nabla_{v_k} \log \tF\right) \tF 
        + \int_{(\R^3)^{2N}} \frac{1}{m_k} \tQ_{ij}(\tF)\,
                 \abs{\nabla_{v_k} \log \tF}^2 \\
        &= \frac{1}{m_k}\int_{(\R^3)^{2N}} 
           2\,\nabla_{v_k} 
           \left ( \frac{\tQ_{ij}(\tF)}{\tF}\right ) \cdot 
           \nabla_{v_k} \tF   
           + \tQ_{ij}(\tF)\abs{\nabla_{v_k} \log \tF}^2.
    \end{align*}
    Since $\tF$ is factorized, $\tQ_{ij}(\tF)/\tF$ does not depend on $v_k$ while $\nabla_{v_k} \log \tF$ only depends on $v_k$. As a consequence, the above integrals vanish, which proves the claim.

    Finally, the comparison principle~\eqref{eq:DissInequality} follows from the same argument as in the single-species case:
    \begin{align*}
        \frac{\d}{\d t} \mathcal{I}(F)\big|_{t=0}=\sum_{i=1}^N \big\langle \frac{1}{m_i}I '(f_i), Q_{ij}(f_i, f_j) \big\rangle 
        &= \sum_{i=1}^N \big\langle \frac{1}{m_i}I '(f_i), \partial_t  \Pi_i \tF\big|_{t=0}\big\rangle \\
        &=  \partial_t \left(\sum_{i=1}^N\frac{1}{m_i}I (\Pi_i \tF) -\frac12 \tI(\tF) \right)\Big|_{t=0} + \partial_t \frac12 \tI(\tF)\big|_{t=0}\\
        &\leq \partial_t \frac12 \tI(\tF)\big|_{t=0},
    \end{align*} where the first term in the second line is non-positive due to \cref{lem:Fisher Carlen}.
\end{proof}

Let $\tF$ be a solution to~\eqref{eq:multi_Landau}. At time $t=0$, the initial data is factorized as $\tF^\circ = (f_1^\circ \otimes \cdots \otimes f_N^\circ)^{\otimes 2}$, and \cref{lem:linearity} allows us to rewrite each contribution to the dissipation as
\begin{equation*}
    \partial_t \tI (\tF)\big|_{t=0} 
    = \sum_{\substack{1 \le i \le N\\N+1 \le j \le 2N}} \big\langle \tI'(\tF^\circ),\tQ_{ij}(\tF^\circ) \big\rangle = \sum_{\substack{1 \le i \le N\\N+1 \le j \le 2N}} \big\langle \tI'_{ij}\left(f_i^\circ \otimes f_j^\circ\right),\tQ_{ij}\left(f_i^\circ \otimes f_j^\circ\right) \big\rangle.
\end{equation*}
\begin{lemma}[Two-particle dissipation identity]\label{lem:two-particle-dissipation}
    Let $1 \le i \le N$ and $N+1 \le j \le 2N$, and let $f_i,f_{j-N}\, (\coloneqq f_j)$ be solutions of \eqref{eq:multi_Landau} given by \cref{lem:local_well_posedness}. The Gâteaux derivative of the two-particles weighted Fisher information $\tI_{ij}(f_i \otimes f_j)$ in the direction $\tQ_{ij}(f_i \otimes f_j)$ is given by
    \begin{equation}\label{eq:T111}
        \big\langle \tIij'\fij, \tQ_{ij}\fij \big\rangle 
        = -2c_{ij}\int_{(\R^3)^{2}} 
          \!\!\!\bbnabla \left(\Aij\, \notbbnabla \Uij\right) 
          : \bbnabla \notbbnabla \Uij.
    \end{equation}
\end{lemma}
\begin{proof}
    We have 
    \begin{align*}
    \big\langle \tIij'\fij, \tQ_{ij}\fij \big\rangle 
    &= \int_{(\R^3)^{2}} 
        2 \left(\bbnabla \left ( \frac{\tQ_{ij}\fij}{ f_i \otimes f_j }\right) 
        \cdot \bbnabla\, \Uij\right) \fij\\
    &\quad + \int_{(\R^3)^{2}} \tQ_{ij}\fij \abs{\bbnabla\, \Uij}^2 
    \\ &\eqqcolon \text{T}_1 + \text{T}_2.
\end{align*}
    Throughout the proof, all integrations by parts are justified the regularity provided by \cref{lem:local_well_posedness}; we will not comment on this in the sequel. 
    
    We begin with the term $\text{T}_2$, integrating by parts, we obtain
    \begin{align*}
        \text{T}_2 &= \int_{(\R^3)^{2}} \tQ_{ij}\fij\, \abs{\bbnabla\, \Uij}^2 \\
            &= \int_{(\R^3)^{2}} \notbbnabla \cdot \left(\cAij \,\notbbnabla \fij\right) 
               \abs{\bbnabla\, \Uij}^2 \\
            &= -2c_{ij}\int_{(\R^3)^{2}} \bbnabla\notbbnabla \Uij 
                : \left(\bbnabla\, \Uij\otimes \Aij \notbbnabla \Uij \right) \fij.
    \end{align*}
    Next we treat the term $\text{T}_1$. Writing $A= A(v_i-v_j)$ we have the identity
    \begin{align*}
        \frac{\tQ_{ij}\fij}{ f_i \otimes f_j } 
        = c_{ij} \lp \notbbnabla \cdot \left(A \notbbnabla \Uij\right) + \notbbnabla \Uij \cdot A \notbbnabla \Uij \rp ,
    \end{align*}
    which allows us to decompose $\text{T}_1$ as
    \begin{align*}
        \text{T}_1
            &= 2c_{ij}\int_{(\R^3)^{2}} 
               \left(\bbnabla \notbbnabla \cdot \left( A \notbbnabla \Uij\right) 
               \cdot \bbnabla\, \Uij\right) \fij \\
            &\quad + 2c_{ij}\int_{(\R^3)^{2}} 
               \left(\bbnabla (\notbbnabla \Uij \cdot A \notbbnabla \Uij) 
               \cdot \bbnabla\, \Uij\right) \fij \\
            &\eqqcolon T_{1,1} + T_{1,2}.
    \end{align*}
    Integrating by parts, we obtain
    \begin{align*}
        \text{T}_{1,1}  &= 2c_{ij}\int_{(\R^3)^{2}} 
            \bbnabla \left(\notbbnabla \cdot \left( A \notbbnabla \Uij\right) \right) 
            \cdot \bbnabla \fij \\
            &= -2c_{ij}\int_{(\R^3)^{2}} 
                \bbnabla \left( A \notbbnabla \Uij\right) : \bbnabla \notbbnabla \fij \\
            &= -2c_{ij}\int_{(\R^3)^{2}} 
                \bbnabla \left( A \notbbnabla \Uij\right) : \bbnabla \notbbnabla \Uij \fij \\
            &\quad -2c_{ij}\int_{(\R^3)^{2}} 
                \bbnabla \left( A \notbbnabla \Uij\right) 
                : \left(\bbnabla\, \Uij\, \otimes \notbbnabla \Uij\right) \fij,
    \end{align*}
    where we used
    \begin{equation*}
        \frac{\bbnabla \notbbnabla \fij}{f_i \otimes f_j}
        = \bbnabla \notbbnabla \Uij 
          + \bbnabla\, \Uij\, \otimes \notbbnabla \Uij.
    \end{equation*}
    By the product rule,
    \begin{align*} 
        \text{T}_{1,2} &= 2c_{ij}\int_{(\R^3)^{2}} 
          \bbnabla \left(A \notbbnabla \Uij\right) 
          : \left( \bbnabla\, \Uij \otimes \notbbnabla\Uij \right) \fij \\
            &\quad + 2c_{ij}\int_{(\R^3)^{2}} \bbnabla\notbbnabla \Uij 
                : \left(\bbnabla\, \Uij\otimes A \notbbnabla \Uij \right) \fij.
    \end{align*}
    Collecting the identities for $\text{T}_2$, $\text{T}_{1,1}$ and $\text{T}_{1,2}$ yields~\eqref{eq:T111}.
\end{proof}

Now we introduce two new physical variables, namely the relative velocity and the center-of-momentum frame; for a fixed pair $(i, j)$ we set
\begin{equation} \label{eq:new_variables}
    z \coloneqq v_i - v_j, \quad z^* \coloneqq \frac{\mi v_i + \mj v_j}{\Mij}, \quad \Mij \coloneqq \mi + \mj, \quad \aij \coloneqq \frac{1}{\mi} + \frac{1}{\mj}.
\end{equation}
Then the transformation $(v_i, v_j) \mapsto (z, z^*)$ is volume-preserving and we obtain
\begin{equation*}
    \bbnabla \mapsto \left(\frac{\gradz}{\sqrt{\mi}} + \frac{\sqrt{\mi}}{\Mij}\gradzs, -\frac{\gradz}{\sqrt{\mj}} + \frac{\sqrt{\mj}}{\Mij}\gradzs\right) , \quad \notbbnabla \:\,\mapsto \aij \gradz.
\end{equation*}
Although the function $\ofij(z, z_*) \coloneq \fij(v_i, v_j)$ is no longer a tensor product in these variables, we obtain convenient formulas such as
\begin{align*}
    \bbnabla \left(\Aij \notbbnabla \Uij \right) = \aij&\Bigg(\frac{\gradz}{\sqrt{\mi}} \left( \Az \gradz \log\ofij \right) + \frac{\sqrt{\mi}}{\Mij} \gradzs\gradz \log\ofij \Az, \\  
    & \quad-\frac{\gradz}{\sqrt{\mj}} \left( \Az \gradz \log\ofij \right) + \frac{\sqrt{\mj}}{\Mij} \gradzs\gradz \log\ofij \Az \Bigg)
\end{align*}
and
\begin{multline*}
    \bbnabla \notbbnabla \log\fij 
    \\ = \aij \left(\frac{\gradz^2}{\sqrt{\mi}} \log\ofij + \frac{\sqrt{\mi}}{\Mij}\gradzs\gradz \log\ofij \Az, 
    -\frac{\gradz^2}{\sqrt{\mj}} \log\ofij + \frac{\sqrt{\mj}}{\Mij}\gradzs\gradz \log\ofij \Az\right).
\end{multline*}
Therefore,
\begin{align*}
    \big\langle \tIij'\fij, \tQ_{ij}\fij \big\rangle &= -2c_{ij}\aij^3 \int_{(\R^3)^2} \gradz^2 \log\ofij : \left(\gradz \left(\Az \nabla_z \log\ofij\right)\right) \ofij \numberthis\label{eq:main term paring} \\
        &\quad -2c_{ij}\frac{\aij^2}{\Mij} \int_{(\R^3)^2} \gradzs \gradz \log\ofij : \left(\gradzs \gradz \log\ofij \Az\right) \ofij. \numberthis\label{eq:quadratic term pairing} \\
        &\eqqcolon \text{I}_{ij} + \text{II}_{ij}.
\end{align*}

Note that the second term \eqref{eq:quadratic term pairing} is non-positive since the symmetric matrix $A$ is positive semi-definite;
\begin{equation*}
    \big\langle \tIij'\fij, \tQ_{ij}\fij \big\rangle \le \text{I}_{ij}.
\end{equation*}
In the regime $-2 \ge \gamma \ge -\sqrt{8}$, we also show that $\text{I}_{ij}$ is non-positive.

\begin{definition} \label{def:optimal_constants}
We define the optimal constant $\Lambda_3$ in the log-Sobolev inequality, for any function $\psi$,
\begin{align} \label{eq:logSob}
    \int_{\S^2} \Gamma_2(\log \psi, \log \psi)\psi \,\ud{\omega} \ge \Lambda_3\int_{\S^2} \abs{\snabla \log\psi}^2\psi \d {\omega},
\end{align}
and $\Lambda_3^{\mathrm{sym}}$ as the optimal constant under the additional symmetry constraint $\psi(z) = \psi (-z)$. Note that $\Lambda_3=2$ and $\Lambda_3^{\mathrm{sym}} \geq 5.5$, see~\cite[Remark 1.4 and Theorem 1.1]{Ji24}). 
\end{definition} 

\begin{proposition}[Coercivity of the two-species dissipation]\label{prop:coercivity}
    In the setting of \cref{lem:two-particle-dissipation},
    \begin{equation*}
        D_{\tI_{ij}, \tQ_{ij}}(f_i \otimes f_j) \coloneqq - \big\langle \tI'_{ij}\fij,\tQ_{ij}\fij \big\rangle
    \end{equation*}
    whenever $-2 \ge \gamma \ge -\sqrt{8}$. Moreover, setting $C_{ij} \coloneqq \Lambda_3 \1_{j \ne N+i} + \Lambda_3^{\mathrm{sym}} \1_{j = N+i}$, we have the lower bound
    \begin{align*}
        D_{\tI_{ij}, \tQ_{ij}}\fij 
            &\ge 2c_{ij}\aij^3 \int_{({\R^3})^{2}} |z|^{2+\gamma} \abs{\Pperp{z}\gradz\left(\hz \cdot \gradz \log\ofij\right) + \frac{\gamma}{2|z|}\Pperp{z}\gradz \log\ofij}^2 \ofij\,\ud{z}\ud{z^*}\\
            &\quad + 2c_{ij}\aij^3 \left( C_{ij} - \frac{\gamma^2}{4} \right) \int_{({\R^3})^{2}} |z|^\gamma \abs{\Pperp{z}\gradz \log\ofij(z, z^*)}^2 \ofij(z, z^*)\,\ud{z}\ud{z^*} \\
            &\quad + 2c_{ij}\frac{\aij^2}{\Mij} \int_{(\R^3)^2} \gradzs \gradz \log\ofij(z, z^*) : \left(\gradzs \gradz \log\ofij(z, z^*) \Az\right) \ofij(z, z^*)\,\ud{z}\ud{z^*}.
    \end{align*}
\end{proposition}
\begin{proof}
    To compute the contraction in \eqref{eq:main term paring}, namely $\text{I}_{ij}$, we use the notation of Appendix~\ref{sec:appendixA};
    \begin{equation*}
        \nabla_z \coloneqq \hz \rpartial + \frac{1}{|z|} \hznabla, \quad \hz \coloneq \frac{z}{|z|}, \quad \rpartial \coloneqq \hz \cdot \nabla_z, \quad \hznabla \coloneqq |z|\Pperp{z}\nabla_z.
    \end{equation*}
    We denote $\tests(r, \omega, z^*) \coloneqq \ufij(r, \omega, z^*) \coloneqq \ofij(z, z^*) \eqqcolon \testf(z, z^*)$ in spherical variables $(z, z^*) = (r\omega, z^*)$ for readability. Combining \cref{lem:decomposition Euclidean Hessian} with \cref{lem:decomposition Hessian-ish}, recalling that $\hz \cdot \hznabla = 0$ and $\Pperp{z}\hz = 0$, we obtain
    \begin{align*}
        \nabla_z^2\log\testf &: \nabla_z\left(|z|^{2+\gamma}\Pperp{z}\nabla_z \log\testf\right) \\
            &= |z|^{\gamma} \abs{\hznabla\rpartial \log\testf}^2 + \left(\gamma-1\right)|z|^{\gamma-1} \hznabla\rpartial \log\testf \cdot \hznabla \log\testf - \gamma\,|z|^{\gamma-2} \abs{\hznabla \log\testf}^2
            \\ &\quad + |z|^{2+\gamma}\norm{\Pperp{z} \nabla^2_z \log\testf\, \Pperp{z} - \frac{1}{|z|}\rpartial \log\testf \Pperp{z}}^2_{\mathrm{HS}} \\
            &\quad+ |z|^{1+\gamma} \rpartial \log\testf\, \Pperp{z} : \left(\Pperp{z} \nabla^2_z \log\testf\, \Pperp{z} 
                - \frac{1}{|z|}\rpartial \log\testf\, \Pperp{z}\right).
    \end{align*}
    Changing to spherical variables $(z, z^*) \mapsto (r, \omega, z^*)$ yields
    \begin{align}
        &\frac{1}{2c_{ij}\alpha_{ij}^3}\,\text{I}_{ij} \nonumber
        \\ &\quad = -\int_{\R^3}\int_{\S^2}\int_0^\infty
            \left(
            r^{2+\gamma} \abs{\nabla_{\S^2}\partial_r \log \ufij}^2
            + \left(\gamma-1\right) r^{1+\gamma} \nabla_{\S^2}\partial_r \log \ufij \cdot \nabla_{\S^2} \log \ufij 
            - \gamma\, r^{\gamma}\abs{\nabla_{\S^2} \log \ufij}^2 \right.
            \label{eq:J1-11}\\
        &\quad\qquad\qquad\qquad\qquad
            \left.+\,r^{\gamma}\norm{\snabla^2 \log \ufij}^2_{\mathrm{HS}}
            + r^{\gamma+1}\partial_r \log \ufij\,
              \Delta_{\S^2} \log \ufij
           \right) \ufij\,\ud r\,\ud\omega\,\ud z^*, \label{eq:J1-12}
    \end{align}
    courtesy of \cref{lem:spherical Hess and Delta}. We now complete the square in \eqref{eq:J1-11} after integrating by parts the last term \eqref{eq:J1-12}. Integrating by parts in spherical variables gives
    \begin{align*}
        -\int_{\S^2} r^{1+\gamma} \left(\partial_r \log\tests \Delta_{\S^2}\log\tests\right) \tests\, \ud{\omega} 
        &= \int_{\S^2} r^{1+\gamma}\left(\snabla\partial_r \log\tests \cdot \snabla\log\tests\right) \tests\, \ud{\omega}  + \int_{\S^2} r^{1+\gamma}\partial_r \tests \abs{\snabla \log\tests}^2 \ud{\omega} \\
        &= - \int_{\S^2} r^{1+\gamma}\left(\snabla\partial_r \log\tests \cdot \snabla\log\tests\right) \tests\, \ud{\omega} \numberthis\label{eq:-1 nabla omega nabla r cdot nabla omega}\\
        &\quad\, + \int_{\S^2} r^{1+\gamma}\partial_r \left( \tests \abs{\snabla \log\tests}^2 \right) \ud{\omega}. \numberthis\label{eq:laplacian to partial r}
    \end{align*}
    The term \eqref{eq:-1 nabla omega nabla r cdot nabla omega} will contribute to the completed square. For \eqref{eq:laplacian to partial r}, an integration by parts in the radial variable gives
    \begin{align*}
        &\int_{\R^3}\int_{S^2}\int_0^\infty r^{1+\gamma} \partial_r \left( \abs{\snabla \log\tests}^2\tests \right) \ud{r}\ud{\omega}\ud{z_*}  \\
        &= \int_{\R^3}\int_{\S^2} \left[ r^{1+\gamma}\left(\abs{\snabla \log\tests}^2\tests\right)(r,\omega,z^*) \right]_0^\infty \ud{\omega}\ud{z^*} 
             -\left(1+\gamma\right) \int_{\R^3}\int_{\S^2}\int_0^\infty r^\gamma \abs{\snabla \log\tests}^2\tests \,\ud{r}\ud{\omega}\ud{z^*}. \numberthis\label{eq:2 nabla omega square}
    \end{align*}
    Finiteness of the Fisher information and the control on the moments shows that the boundary term at $r \to +\infty$ vanishes. At $r \to 0$, the boundary term vanishes as well, since, in particular, $\gamma > -3$.
    
    We then complete the square by combining \eqref{eq:J1-11} with  \eqref{eq:-1 nabla omega nabla r cdot nabla omega} and \eqref{eq:2 nabla omega square};
    \begin{multline*}\label{eq:complete the square}
        r^{2+\gamma} \abs{\snabla\partial_r \log\tests}^2 + \gamma r^{1+\gamma} \snabla\partial_r \log\tests \cdot \snabla \log\tests + r^{\gamma} \abs{\snabla \log\tests}^2 \\
            = r^{2+\gamma}\abs{\snabla \partial_r \log\tests + \frac{\gamma}{2r}\snabla \log\tests}^2 + \left(1 - \frac{\gamma^2}{4}\right)r^\gamma\abs{\snabla \log\tests}^2.
    \end{multline*}
    Inserting this into \eqref{eq:J1-11}, we obtain
    \begin{align*}
        \frac{1}{2c_{ij}\aij^3}J_1 
            = &- \int_{\R^3}\int_0^\infty \int_{\S^2} r^{2+\gamma}\abs{\snabla \partial_r \log\tests + \frac{\gamma}{2r}\snabla \log\tests}^2 \tests\, \ud{\omega}\ud{r}\ud{z^*} \\
                &- \int_{\R^3}\int_0^\infty \int_{\S^2} r^\gamma\left(\Gamma_2\left(\log\tests, \log\tests\right) -\frac{\gamma^2}{4}\abs{\snabla \log\tests}^2\right) \tests\, \ud{\omega}\ud{r}\ud{z^*},
    \end{align*}
    recalling the iterated carré du champ $\Gamma_2$ on for the spherical Laplacian on $\S^2$ given by \eqref{eq:Gamma 2}.
    
    Unfortunately, $\ufij(r, \cdot, z^*) \eqqcolon g(r, \cdot, z^*)$ is not necessarily symmetric on $\S^2$ for $j \ne N+i$, hence the optimal constant in the log-Sobolev inequality is $\Lambda_3 = 2$ in that case (see e.g. \cite[Remark 1.4]{Ji24});
    \begin{equation*}
        \int_{\S^2} \Gamma_2(\log \tests, \log \tests)\tests \,\ud{\omega} \ge \Lambda_3\int_{\S^2} \abs{\snabla \log\tests}^2\tests \,\ud{\omega}.
    \end{equation*}
    In the case of symmetry, that is $j = N+i$, then the optimal constant becomes $\Lambda_3^{\mathrm{sym}} \ge 5.5$ \cite[Theorem 1.1]{Ji24}.
\end{proof}
\begin{remark}
    Without invoking the log-Sobolev inequality, the sign of the dissipation would remain unclear; indeed, from the proof of \cref{prop:coercivity}, it follows that
\begin{align*} 
    D_{\tI_{ij}, \tQ_{ij}}\fij &= 2c_{ij}\aij^3 \int_{({\R^3})^{2}} |z|^{2+\gamma} \abs{\Pperp{z}\gradz\left(\hz \cdot \gradz \log\ofij\right) + \frac{\gamma}{2|z|}\Pperp{z}\gradz \log\ofij}^2 \ofij\,\ud{z}\ud{z^*}\\
        &\quad + 2c_{ij}\aij^3 \int_{({\R^3})^{2}} |z|^{2+\gamma} \left(\norm{\left(\Pperp{z}\gradz\right)^2 \log\ofij}^2_{\mathrm{HS}} - \frac{\gamma^2}{4|z|^2}\abs{\Pperp{z}\gradz \log\ofij}^2 \right) \ofij\,\ud{z}\ud{z^*} \\
        &\quad + 2c_{ij}\frac{\aij^2}{\Mij} \int_{(\R^3)^2} \gradzs \gradz \log\ofij : \left(\gradzs \gradz \log\ofij \Az\right) \ofij\,\ud{z}\ud{z^*}, \numberthis\label{eq:J2 z z star in dissipation equality}
\end{align*}
where in the second line we used \eqref{eq:Gamma 2 HS}. In the original variables
\begin{equation*}
    v_i = \frac{m_j}{\Mij} z + z^*, \quad v_j = -\frac{m_i}{\Mij} z + z^*, \quad \gradz \mapsto \frac{1}{\aij}\notbbnabla\,, \quad \gradzs \mapsto \nablaijplus,
\end{equation*}
and thus
\begin{align*}
    D_{\tI_{ij}, \tQ_{ij}}\fij &= 2\frac{c_{ij}}{\aij} \int_{({\R^3})^{2}} |v_i-v_j|^{2+\gamma} \left|\Pperp{v_i - v_j}\notbbnabla\left(\frac{v_i - v_j}{|v_i - v_j|} \cdot \notbbnabla \log\fij\right)\right. \\
        &\hphantom{= 2\frac{c_{ij}}{\aij} \int_{({\R^3})^{2}} |v_i-v_j|^{2+\gamma}} \quad + \left.\frac{\aij \gamma}{2|v_i - v_j|}\Pperp{v_i - v_j}\notbbnabla \log\fij\right|^2 \fij\\
        &\quad + 2\frac{c_{ij}}{\aij} \int_{({\R^3})^{2}} |v_i - v_j|^{2+\gamma} \left(\norm{\left(\Pperp{v_i - v_j}\notbbnabla\, \right)^2 \log\ff}^2_{\mathrm{HS}}\right. \\
        &\hphantom{\quad + 2\frac{c_{ij}}{\aij} \int_{({\R^3})^{2}} |v_i - v_j|^{2+\gamma}} \quad - \left.\frac{\aij^2 \gamma^2}{4|v_i - v_j|^2}\abs{\Pperp{v_i - v_j}\notbbnabla \log\fij}^2 \right) \fij \\
        &\quad + 2\frac{c_{ij}}{\Mij} \int_{(\R^3)^2} \nablaijplus \notbbnabla \log\fij \\
        &\hphantom{\quad + 2\frac{c_{ij}}{\Mij} \int_{(\R^3)^2} \nablaijplus } : \left(\nablaijplus \notbbnabla \log\fij \Aij\right) \fij.
\end{align*}
\end{remark}

\subsection{Local well-posedness} \label{subsec:local}

To derive global-in-time well-posedness of the multi-species Landau equation from the monotonicity of the Fisher information, we first establish local well-posedness. The proof is similar to the single-species case, and therefore we are brief on some of the details.

\begin{lemma}[Local well-posedness] \label{lem:local_well_posedness}
    Let $l\geq 5$ and $f^\circ_i\in L^2_l\cap L\log L$ be probability densities for $1\leq i \leq N$. Then for some $T>0$ depending only on $\|f_i^\circ\|_{L^2_l}$, $\mathcal{H}(F_N^\circ)$,  there exists a unique strong solution $f_i\in L^\infty([0,T],  L^2_l)\cap C^\infty((0,T)\times \R^3)$ to the system
    \begin{align} 
	\partial_t f_i &= \sum_{j=1}^N Q_{ij}(f_i,f_j), \quad f_i(0,v)= f_i^\circ(v), 
    \end{align}
    which satisfies
    \begin{align}
        \sup_{t\in [0,T]} \sum_{i=1}^N\| f_i(t)\|^2_{L^2_l} + c \sum_{i=1}^N\int_0^T \| f_i\|^2_{H^1_{l+\gamma}} &\leq C_T \sum_{i=1}^N \| f_i^\circ\|^2_{L^2_l}.
    \end{align}
    Moreover, the entropy is non-increasing, 
    \begin{align}\label{def:DDissipation}
        \frac{\d}{\d t} \mathcal H (F_N) &= -\mathcal{D}(F_N) =  - \sum_{i,j=1}^N D_{ij}(f_i,f_j),
    \end{align} where 
    \begin{equation*}
         D_{ij}(f_i,f_j) \coloneqq -\frac12 \int_{(\R^3)^2} \left(\left( \Pi^\perp_{v_i-v_j} |v_i-v_j|^{2+\gamma}\notbbnabla  \log \fij\right) \cdot \notbbnabla  \log \fij\right) \fij \ud{v_i} \ud{v_j} . 
    \end{equation*}
\end{lemma}
\begin{proof}
    The existence of the solution follows by a standard regularization argument, along with the conservation of mass, total momentum and energy. For the regularity estimate, we consider the time derivative
    \begin{align*}
        \frac{\d}{\d t} \frac12 \int_{\R^3} f_i^2(v) \langle v\rangle^{l}\, \ud{v} &= \sum_{j=1}^N \int_{\R^3} f_i(v)\, Q_{ij}(f_i,f_j)\, \langle v\rangle^{l}\, \ud{v}.
    \end{align*}
    We then estimate each term in the sum separately. By definition of $Q_{ij}$, we have
    \begin{align*}
        \int_{\R^3} f_i(v) &Q_{ij}(f_i,f_j) \langle v\rangle^{l}\, \ud{v} \\
        &\, = \int_{\R^3} f_i(v_i) \, \frac{\gradvi}{\mi}\, \cdot \left( \int_{\R^3} c_{ij}\, A(v_i-v_j)\,
        \nablaijminus
        \fij(v_i,v_j)\,\ud v_j \right) \langle v_i\rangle^{l}\, \ud{v_i}\\
    & \, =-\int_{\R^3}   
        \mathcal{A}_{ij} [f_j] : \left(\frac{{\nabla f_i}}{\mi}\right)^{\otimes 2} \langle v\rangle^{l}\, \ud{v}  -\int_{\R^3}   
        \mathcal{A}_{ij} [f_j] : \left(\frac{\nabla f_i}{\mi} \otimes \frac{\nabla \left(\langle v\rangle^{l}\right)}{\mi}\right)f_i\, \ud{v} \\
        &\quad\, + \int_{\R^3}  
        \left(\frac{\nabla}{\mj} \cdot \mathcal{A}_{ij} [f_j]\right) \cdot  \frac{\nabla}{\mi}\left(f_i\, \langle v\rangle^{l}\right) f_i\, \ud{v}\\
    & \, \eqqcolon  -\text{I}_{ij} + \text{II}_{ij} + \text{III}_{ij}.
    \end{align*}
    where the operator $\mathcal{A}_{ij}$ is defined by
    \begin{align*}
        \mathcal{A}_{ij}[g](v) \coloneqq \int_{\R^3} c_{ij} \Pi_{v-v^*}^\perp |v-v^*|^{2+\gamma} g(v^*) \d v^*.
    \end{align*}
    Since $f_j$ are probability densities whose energy and entropy are bounded by the total energy and entropy of the system, we have the classical estimate, see \cite{DV00},
    \begin{align*}
        \text{I}_{ij} \geq c \int_{\R^3} |\nabla f_i(v)|^2 \langle v \rangle^{ l +\gamma} \,\ud{v},
    \end{align*}
    where $c>0$ depends on the total energy and entropy of the initial data. Using standard interpolation estimates, making use of $l\geq 5$, we bound
    \begin{align*}
        |\text{II}_{ij}| + |\text{III}_{ij}| \leq \frac{1}{4} \sum_{j=1}^N\text{I}_{ij} + C \big(1+\sum_{j=1}^N \| f_j\|_{L^2_l} \big)^\beta \| f_i\|^2_{L^2_l},
    \end{align*}
    where $\beta>0$ is a constant depending on $\gamma$. For the quantity
    \begin{align*}
        y(t) = \sum_{i=1}^N \int_{\R^3} f_i^2(v) \langle v\rangle^{l} \ud{v},
    \end{align*}
    this yields the differential inequality
    \begin{align*}
        \frac{\d}{\d t} y(t) + \frac{3}{4} \sum_{i,j=1}^N \text{I}_{ij} \leq C \langle y \rangle ^{1+\beta}(t).
    \end{align*}
    Hence, for some $T>0$ we obtain the claimed estimate.
\end{proof}

\begin{corollary}[Production of Fisher information]  \label{cor:Fisher}  Let $l\geq 7$ and the initial data $f^\circ_i\in L^2_l\cap L\log L$ be probability densities for $1\leq i \leq N$, and $f_i\in L^\infty([0,T],  L^2_l)\cap C^\infty((0,T)\times \R^3)$ be the solution from~\cref{lem:local_well_posedness}. Then $t\in(0,T)$ a.e. we have for all $1\leq i\leq N$,
\begin{align*}
    \| f_i\|_{H^2_k} \leq C_T<\infty, \qquad \text{and} \qquad 
    \mathcal{I}(f_i) \leq C_T< \infty,
\end{align*}
in particular the Fisher information $\mathcal{I}(F)$ is finite.
\end{corollary}
\begin{proof}
    \cref{lem:local_well_posedness} shows that almost everywhere in $(0,T)$ we have 
    \begin{align*}
        \| \nabla f_i(t^*)\|_{L^2_{l+\gamma}} < \infty, \quad \text{for all } 1\leq i\leq N. 
    \end{align*}
    We now derive energy estimates for $\partial_k f_i$ for any $k\in \{1,2,3\}$. The equation for the partial derivative reads
    \begin{align*}
        \partial_t \partial_k f_i = \sum_{i=1}^N Q_{ij}(\partial_k f_i, f_j) + Q_{ij}( f_i, \partial_k f_j).
    \end{align*}
    We apply the same argument as in~\cref{lem:local_well_posedness}, which yields
    \begin{align*}
        \frac12 \frac{\d}{\d t} \int_{\R^3} (\partial_k f_i)^2(v) \langle v \rangle^ {l} \d v &= \sum_{j=1}^N\int_{\R^3} \partial_k f_i(v) \left( Q_{ij}(\partial_k f_i,f_j) + Q_{ji} (f_i,\partial_k f_j)\right)\langle v \rangle^ {l} \d v.
    \end{align*}
    As before, we can separate the contributions from the individual species $1\leq j\leq N$. Using the notation from above, we obtain
    \begin{align*}
        \int_{\R^3} \partial_k f_i(v) \left( Q_{ij}(\partial_k f_i,f_j) \right. &+ \left.Q_{ji} (f_i,\partial_k f_j)\right)\langle v \rangle^ {l} \d v = - \int_{\R^3}   \frac{\nabla \partial_kf_i}{m_i} \mathcal A [f_j] \frac{\nabla \partial_kf_i}{m_i} \langle v_i \rangle^{l}  \d v_i \\
        &-\int_{\R^3}   \frac{\nabla \partial_kf_i}{m_i} \mathcal A [f_j] \partial_kf_i \frac{\nabla }{m_i}\langle v_i \rangle^{l}  \d v_i +\int_{\R^3}    \partial_kf_i \frac{\nabla_{v_j}}{m_j} \cdot \mathcal A [f_j] \frac{\nabla }{m_i} \big(\partial_kf_i \langle v_i \rangle^{l} \big)  \d v_i \\
        &-\int_{\R^3}   \frac{\nabla f_i}{m_i} \mathcal A [\partial_ kf_j]  \frac{\nabla }{m_i}\big( \partial_k f_i\langle v_i \rangle^{l} \big) \d v_i +\int_{\R^3}   f_i\frac{\nabla_{v_j}}{m_j} \cdot \mathcal A [\partial_ kf_j]  \frac{\nabla }{m_i}\big( \partial_k f_i\langle v_i \rangle^{l} \big) \d v_i\\
        &\eqqcolon - \text{I}_{ij} + \text{II}_{ij}+ \text{III}_{ij}+ \text{IV}_{ij}+ \text{V}_{ij}.
    \end{align*}
    We again use the lower bound for the matrix $A$ in terms of total energy and entropy, and obtain the bound
    \begin{align*}
        \text{I}_{ij} \geq c \| \partial_k \nabla f_i \|^2_{L^2_{l+\gamma}}.
    \end{align*}
    The remaining terms can be bounded by straightforward interpolation arguments to get the total bound
    \begin{align*}
        \partial_t \sum_{i=1}^N \|\nabla f_i\|^2_{L^2_l} + c \sum_{i=1}^N \|\nabla^2 f_i\|^2_{L^2_{l+\gamma}} \leq C \big( 1+ \sum_{i=1}^N \|\nabla f_i\|^2_{L^2_l} \big)^{1+\beta},
    \end{align*}
    for some constant $\beta>0$. The upper bound for the solution of the corresponding ODE yields, after possibly choosing a smaller $T>0$, the bound
    \begin{align*}
        \int_0^T \| \nabla^2 f_i\|^2_{L^2_{l+\gamma}} \d t \leq C_T < \infty.
    \end{align*}
    To conclude the finiteness of the Fisher information, we recall that for any $k>3$ we can estimate
    \begin{align*}
        I(f) = 4 \| \sqrt{f}\|^2_{\dot{H}^1} \leq C_k  \| f\|_{H^2_k} ,
    \end{align*}
    see~\cite[Lemma~1]{To00} (observe the slightly different convention for $\|\cdot \|_{L^2_l}$ therein). This concludes the proof.
\end{proof}
This local result is complemented by a control of the growth of the moments of the solution. 

\begin{lemma} [Growth bound for the moments] \label{lem:Moments}
        Let $l\in\mathbb{N}$ and the initial data $f^\circ_i\in L^2_l\cap L^1_{2l}\cap L\log L$ be probability densities and let $F_N = \lp f_i \rp_{i=1}^N$ be the unique local strong solution from~\cref{lem:local_well_posedness}. Then $f_i(t)$ are probability densities and the total momentum and total energy are conserved,
        \begin{align*}
            M(t) &\coloneqq \sum_{i=1}^N \int_{\R^3} \mi v\, f_i(t,v) \, \ud{v} = M(0), \\
            E(t) &\coloneqq \sum_{i=1}^N \int_{\R^3} \frac12 \mi |v|^2\, f_i(t,v) \, \ud{v} = E(0).
        \end{align*}
    The higher-order moments satisfy the  growth estimate
    \begin{align} \label{eq:momentbound}
        \| f_i\|_{L^1_l} \leq C_l \sum_{i=1}^N\| f_i^\circ\|_{L^1_l} + C_lt^{1+\lceil{\frac{l-4}{2}\rceil}} .
    \end{align}
\end{lemma}
\begin{remark}
    For $l\leq 4$, the estimate~\eqref{eq:momentbound} coincides with the bound for the single species case, see~\cite{CM17}. On the other hand, for $l>4$ the estimates are likely not sharp.
\end{remark}

\begin{proof}[Proof of~\cref{lem:Moments}]
    Conservation of total mass and momentum follows from the fact that these quantities are preserved by each collision operator $Q_{ij}$. The growth bound for higher-order moments requires some care due to the asymmetry of the multi-species equation. Nevertheless, we can largely follow the strategy of~\cite{CM17}. Let $l\geq 0$, and consider the \emph{total} $l$-th moment
    \begin{align} \label{eq:lMoment}
        M_l(t) = \sum_{i=1}^N \int_{\R^3} m_i \langle v\rangle^l f_i(v)\, \ud{v},  
    \end{align}
    and its time derivative 
    \begin{align*}
        \frac{\d}{\d t} M_l(t) &= \sum_{i,j=1}^N \int_{\R^3} m_i \langle v\rangle^l Q_{ij}(f_i,f_j) \, \ud{v}  \\
        &= \sum_{i=1}^N \int_{\R^3} m_i \langle v\rangle^l\, Q_{ii}(f_i,f_i)\, \ud{v}  + \sum_{i<j} \int_{\R^3} m_i \langle v\rangle^l Q_{ij}(f_i,f_j)+ m_j \langle v\rangle^l Q_{ji}(f_j,f_i)\, \ud{v} \\
        & \eqqcolon \sum_{i=1}^N K_i + \sum_{i<j} K_{ij}.
    \end{align*}
    The terms $K_i$ can be treated as in \cite[Lemma 8]{CM17}, yielding for some $C>0$ 
    \begin{align*}
        \sum_{i=1}^N K_i \leq \mathcal{D}(F_N) M_l + C M_{l-4} +C ,
    \end{align*}
    where $\mathcal{D}$ is the dissipation functional introduced in~\eqref{def:DDissipation}. 
    It remains to estimate the cross-interaction terms $K_{ij}$. Due to the $m_i$-weights in~\eqref{eq:lMoment}, the terms $K_{ij}$, $i\neq j$, take the form
    \begin{align*}
        K_{ij} = - \int_{\R^3\times \R^3} (f_i\otimes f_j ) |v_i-v_j|^{2+\gamma} \left(\Pi^\perp_{v_i-v_j} \notbbnabla \log (f_i \otimes f_j)\right) \cdot \left(\nabla \langle v_i \rangle^l - \nabla  \langle v_j \rangle^l  \right) \d v_i \d v_j.
    \end{align*}
    As in the single-species case, we introduce a smooth, radially symmetric cutoff function $\chi\in C^\infty_c$ satisfying $\1_{B_\frac12} \leq \chi \leq \1_{B_1}$.
  
    We then decompose $K_{ij}$ into a contribution near the singularity and a far-field part:
    \begin{align*}
        K_{ij} &=  - \int_{\R^3\times \R^3} (f_i\otimes f_j )\chi(v_i-v_j) |v_i-v_j|^{2+\gamma} \Pi^\perp_{v_i-v_j} \notbbnabla \log (f_i \otimes f_j) \cdot \left(\nabla \langle v_i \rangle^l - \nabla  \langle v_j \rangle^l  \right) \d v_i \d v_j  \\
        &\quad\, -\int_{\R^3\times \R^3} (f_i\otimes f_j )(1-\chi)(v_i-v_j) |v_i-v_j|^{2+\gamma} \Pi^\perp_{v_i-v_j} \notbbnabla \log (f_i \otimes f_j) \cdot \left(\nabla \langle v_i \rangle^l - \nabla  \langle v_j \rangle^l  \right) \d v_i \d v_j \\
        &=: K_{ij,1} + K_{ij,2}.
    \end{align*}
    The first term can be estimated exactly as in the single-species case, giving
    \begin{align*}
        |K_{ij,1}| \leq 2 \mathcal{D}_{ij} M_l + C l^{-\gamma} M_{l-4}.
    \end{align*}
    To treat the term $K_{ij,2}$, we integrate by parts and obtain
    \begin{align*}
        K_{ij,2} &= \int_{\R^3\times \R^3}  (f_i \otimes f_j)(1-\chi) |v_i-v_j|^{2+\gamma} \notbbnabla \big(\Pi^\perp_{v_i-v_j}   (\nabla \langle v_i \rangle^l - \nabla \langle v_j \rangle^l  \big)  \d v_i \d v_j \\
        &= - \alpha_{ij}\int_{\R^3\times \R^3}  (f_i \otimes f_j)(1-\chi) |v_i-v_j|^{\gamma}   (v_i-v_j)\cdot (\nabla \langle v_i \rangle^l - \nabla \langle v_j \rangle^l) \,\d v_i \d v_j\\
        & \quad +\int_{\R^3\times \R^3}  (f_i \otimes f_j)(1-\chi) |v_i-v_j|^{2+\gamma}   \Pi_{v_i-v_j}^\perp \colon \notbbnabla(\nabla \langle v_i \rangle^l - \nabla \langle v_j \rangle^l) \,    \d v_i \d v_j.
    \end{align*}
    Now the first term in $K_{ij,2}$ is non-positive since 
    \begin{align*}
        (v_i-v_j) (\nabla \langle v_i \rangle^l - \nabla \langle v_j \rangle^l) &\geq 0, 
    \end{align*} while the second term can be bounded by 
    \begin{align*}
        \int_{\R^3\times \R^3}  (f_i \otimes f_j)(1-\chi) |v_i-v_j|^{2+\gamma}   \Pi_{v_i-v_j}^\perp \colon \notbbnabla(\nabla \langle v_i \rangle^l - \nabla \langle v_j \rangle^l)    \d v_i \d v_j 
        \leq C  M_{l-2} 
    \end{align*}
    Collecting terms, we obtain
    \begin{align} \label{eq:momentdiff}
        \frac{\d}{\d t} M_l(t) \leq \mathcal{D}(F_N) M_l +C  M_{l-4} + C M_{l-2} \leq \mathcal{D}(F_N) M_l + C M_{l-2}.
    \end{align}
    For $l\leq 4$, the moment $M_{l-2}$ is uniformly bounded, and therefore
    \begin{align} \label{eq:momentproof}
        M_l(t) \leq C M_l(F_N^\circ) +Ct.
    \end{align}
    For the case $l>4$, it suffices to iteratively insert the estimates obtained from~\eqref{eq:momentproof} back into~\eqref{eq:momentdiff}.
\end{proof}

\subsection{Global well-posedness} \label{subsec:Thm1Proof}
\begin{proof}[Proof of~\cref{thm:Global}]
The global well-posedness follows by combining the results of the previous sections. Let $\ell\geq 7$ and $F_N^\circ$ be the initial datum. The local well-posedness result in~\cref{lem:local_well_posedness} yields a strong solution that can be continued as long as $\|f_i\|_{L^2_l}$ and $H(f_i)$ remain bounded. Since $\mc H(F_N)$ is monotone decreasing in time, it remains to provide a bound for the weighted $L^2$ norm. 

To this end, we observe that~\cref{cor:Fisher} yields finiteness of the Fisher information for small $t>0$, which is non-increasing in time due to~\cref{lem:linearity} and~\cref{prop:coercivity}. Furthermore,~\cref{lem:Moments} gives polynomial bounds for the moments $\|f\|_{L^1_{2l}}$. By interpolation we obtain
\begin{align*}
    \| f\|_{L^2_l} \leq  \|f\|^\frac12_{L^1_{2l}} \|f\|_{L^3}^\frac12 \leq C \|f\|^\frac12_{L^1_{2l}} I^\frac12(f),
\end{align*}
establishing the uniform bound and finishing the proof.
\end{proof}

\section{The two-species Landau system with one species having infinite mass} \label{Sec:Infinite}
\subsection{Dissipation identities and spherical Fisher information} \label{subsec:PrepInfinite}
As advertised in \cref{thm:FisherIncrease}, the standard Fisher information is not a Lyapunov functional for the multi-species Landau system in the limiting case of infinite mass, due to its lack of dissipation under the spherical diffusion operator
\begin{equation*}
    Lf(v) \coloneqq \nabla \cdot \left(|v|^{2+\gamma} \Pperp{v} \nabla f\right)
\end{equation*}
whenever $\gamma < -\sqrt{8}$. We will prove this in the next \cref{sec:nonMono}. However, its (weighted) spherical part
\begin{equation*}
    J_\beta(f) \coloneqq \int_{\R^3} \frac{\abs{\Pperp{v} \nabla \log f}^2}{|v|^\beta} f \,\ud{v}, \quad \beta \ge 0,
\end{equation*}
does dissipate along the flow generated by $Lf$.
\begin{lemma}[Dissipation of the spherical Fisher information along $L$]\label{lem:J-diss along L}
    Let $f$ be a smooth probability density. The dissipation
    \begin{equation*}
        D_{J_\beta, L}(f) \coloneqq - \big\langle J_\beta'(f), Lf \big\rangle
    \end{equation*}
    is given by
    \begin{equation}\label{eq:dissipation J beta L}
        D_{J_\beta, L}(f)
        = 2 \int_{\R^3} |v|^{2+\gamma-\beta} \norm{\left(\Pperp{v}\nabla\right)^2 \log f}^2_{\mathrm{HS}} f \,\ud{v}.
    \end{equation}
\end{lemma}
\begin{proof}
    We first note, in view of \eqref{eq:nabla nabla omega traced pi is spherical Laplacian} and \eqref{eq:nabla omega of nabla omega identity}, using spherical variables $v = r\omega$ with $(r, \omega) \coloneqq (|v|, \hv)$ and $\hv \coloneqq v/|v|$, and denoting $\tf(r, \omega) \coloneqq f(v)$, that
    \begin{equation*}
        Lf = r^\gamma \Delta_{\S^2} \tf, \quad \frac{Lf}{f} = r^\gamma\left(\Delta_{\S^2} \log\tf  + \abs{\snabla \log\tf}^2\right), \quad \tf(r, \omega) \coloneq f(v).
    \end{equation*}
    Hence,
    \begin{align*}
        \big\langle J_\beta'(f) , Lf\big\rangle &= 2 \int_{0} ^\infty \int_{\S^2} r^{\gamma-\beta} \snabla \log\tf \cdot \snabla\left(\Delta_{\S^2} \log\tf + \abs{\snabla \log\tf}^2\right) \tf \ud{\omega}\ud{r} \\
            &\quad + \int_{0}^\infty \int_{\S^2} r^{\gamma-\beta}\abs{\snabla \log\tf}^2 \left(\Delta_{\S^2} \log\tf  + \abs{\snabla \log\tf}^2\right) \tf \ud{\omega}\ud{r}.
    \end{align*}
    Integrating by parts yields
    \begin{equation}\label{eq:pairing j L after IBP}
        \big\langle J_\beta'(f) , Lf\big\rangle = \int_0^\infty \int_{\S^2} r^{2+\gamma-\beta}\left(2\abs{\snabla \log\tf}^{4} + 3\snabla \left(\abs{\snabla \log\tf}^2\right) \cdot \snabla \log\tf - 2\left(\Delta_{\S^2} \log\tf\right)^2\right) \tf \ud{\omega}\ud{r}.
    \end{equation}
    Recall Bochner's formula on $\S^2$;
    \begin{equation*}
        \frac{1}{2}\Delta_{\S^2} \left(\abs{\snabla \log\tf}^2\right) = \norm{\nabla_{\S^2}^2 \log\tf}^2_{\mathrm{HS}} + \abs{\snabla \log\tf}^2 + \snabla \left(\Delta_{\S^2} \log\tf\right) \cdot \snabla \log\tf.
    \end{equation*}
    Multiplying this identity by $\tf$ and integrating by parts, we obtain
    \begin{multline*}
        -\frac{1}{2}\int_{\S^2} \left(\snabla \left(\abs{\snabla \log\tf}^2\right) \cdot \snabla \log\tf\right) \tf \\ = \int_{\S^2} \left(\norm{\nabla_{\S^2}^2 \log\tf}^2_{\mathrm{HS}} + \abs{\snabla \log\tf}^2 + \abs{\snabla \log\tf}^{4} + \snabla\left(\abs{\snabla \log\tf}^2\right) \cdot \snabla \log\tf - \left(\Delta_{\S^2} \log\tf\right)^2\right)\tf,
    \end{multline*}
    whose convenient form is
    \begin{equation}\label{eq:convenient bochner IBP}
        \frac{3}{2}\int_{\S^2} \left(\snabla \left(\abs{\snabla \log\tf}^2\right) \cdot \snabla \log\tf\right) \tf = \int_{\S^2} \left(\left(\Delta_{\S^2} \log\tf\right)^2 - \Gamma_2(\log \tf, \log \tf) - \abs{\snabla \log\tf}^{4}\right)\tf,
    \end{equation}
    and inserting \eqref{eq:convenient bochner IBP} into \eqref{eq:pairing j L after IBP} yields
    \begin{equation}\label{eq:j L dissipation}
        \big\langle J_\beta'(f) , Lf\big\rangle = -2\int_{0}^{\infty} \int_{\S^2} r^{\gamma-\beta} \Gamma_2\left(\log\tf, \log\tf\right) \tf\, \ud{\omega}\ud{r}.
    \end{equation}
    The result follows recalling \eqref{eq:Gamma 2 HS}.
\end{proof}
Of course, the Boltzmann (relative) entropy also dissipates along $L$ as well and its derivative corresponds to the spherical (weighted) Fisher information.
\begin{lemma}[Dissipation of entropy along $L$]\label{lem:H-diss along L}
    Let $f$ be a smooth probability density. The dissipation
    \begin{equation*}
        D_{H, L}(f) \coloneqq - \big\langle H'(f), Lf \big\rangle
    \end{equation*}
    is given by
    \begin{equation*}
        D_{H, L}(f)
        = \int_{\R^3} |v|^{2+\gamma} \abs{\Pperp{v}\nabla \log f}^2 f \,\ud{v}.
    \end{equation*}
\end{lemma}
On the other hand, the expression of the Gateaux derivative of $I$ along $L$ can be deduced, at least informally, from $D_{\tI_{ij}, \tQ_{ij}}(f_i \otimes f_j)$ setting $f_i = f_j = f$, $m_i = 1, m_j = +\infty$, $\gamma = -3$, and noticing that \eqref{eq:J2 z z star in dissipation equality} vanishes. This observation follows from the fact that lifting is equivalent to a symmetrization argument (see~\cite{GGPTZ25} and e.g. the proof of \cref{prop:I-diss along Q}). 
\begin{proposition}[Evolution of the Fisher information along $L$]\label{prop:I prime L}
    Let $f$ be a smooth probability density. Then
    \begin{align*}
        \big\langle I'(f) , Lf\big\rangle = &-2\int_{\R^3} |v|^{2+\gamma}\abs{\Pperp{v}\nabla \left(\hv \cdot \nabla \log f\right) + \frac{\gamma}{2|v|}\Pperp{v}\nabla \log f}^2 f\,\ud{v}\\
        &-2\int_{\R^3} |v|^{2+\gamma} \left(\norm{\left(\Pperp{v}\nabla\right)^2 \log f}_{\mathrm{HS}}^2 - \frac{\gamma^2}{4|v|^2}\abs{\Pperp{v}\nabla \log f}^2 \right) f\,\ud{v}.
    \end{align*}
\end{proposition}
\begin{proof}
    We decompose
    \begin{align*}
        \big\langle I'(f) , Lf\big\rangle &= 2 \int_{\S^2} \int_{0} ^\infty \partial_r \log\tf\, \partial_r \left(r^{\gamma}\left(\Delta_{\S^2} \log\tf  + \abs{\snabla \log\tf}^2\right)\right) \tf\, r^2 \ud{r}\ud{\omega} \\
            &\quad + \int_{\S^2}\int_{0} ^\infty  r^{\gamma}\left(\partial_r \log\tf\right)^2\left(\Delta_{\S^2} \log\tf  + \abs{\snabla \log\tf}^2\right) \tf\, r^2 \ud{r}\ud{\omega} \\
            &\quad + \big\langle J_0'(f) , Lf\big\rangle.
    \end{align*}
    Integrating by parts, we obtain
    \begin{align*}
        2 \int_{\S^2} \partial_r \log\tf\, &\partial_r \left(r^{\gamma}\left(\Delta_{\S^2} \log\tf  + \abs{\snabla \log\tf}^2\right)\right) \tf \ud{\omega} \\
            &= 2\int_{\S^2} \partial_r \log\tf \left(\gamma r^{\gamma-1} + \partial_r\right)\left(\Delta_{\S^2} \log\tf  + \abs{\snabla \log\tf}^2\right)\tf \,\ud{\omega} \\
            &= 2 \int_{\S^2} r^{\gamma} \left(\frac{1}{2} \partial_r \log\tf\, \partial_r \left(\abs{\snabla \log\tf}^2\right) - \abs{\snabla \partial_r \log\tf}^2 - \frac{\gamma}{r} \partial_r \left(\abs{\snabla \log\tf}^2\right)\right) \tf \ud{\omega}
    \end{align*}
    and
    \begin{equation*}
        \int_{\S^2} r^{\gamma}\left(\partial_r \log\tf\right)^2\left(\Delta_{\S^2} \log\tf  + \abs{\snabla \log\tf}^2\right) \tf \ud{\omega} \\
            = - \int_{\S^2} r^{\gamma} \partial_r \log\tf\, \partial_r \left(\abs{\snabla \log\tf}^2\right) \tf \ud{\omega}.
    \end{equation*}
    Completing the square
    \begin{equation*}
        \abs{\snabla \partial_r \log\tf}^2 + \frac{\gamma}{r} \partial_r \left(\abs{\snabla \log\tf}^2\right) = \abs{\snabla \partial_r \log\tf + \frac{\gamma}{2r} \snabla \log\tf}^2 - \frac{\gamma^2}{4r^2}\abs{\snabla \log\tf}^2,
    \end{equation*}
    courtesy of \eqref{eq:j L dissipation}, we are left with
    \begin{align*}
        \big\langle I'(f) , Lf\big\rangle &= -2 \int_{\S^2}\int_{0}^{\infty}  r^{\gamma} \abs{r\snabla\partial_r \log\tf + \frac{\gamma}{2}\snabla \log\tf}^2 \tf\,\ud{r}\ud{\omega}\\
        &\quad -2 \int_{\S^2}\int_{0}^{\infty}  r^{\gamma} \left( \Gamma_2\left(\log\tf, \log\tf\right) - \frac{\gamma^2}{4}\abs{\snabla \log\tf}^2\right) \tf\, \ud{r}\ud{\omega}.
    \end{align*}
\end{proof}
Now, we recall the dissipation of the Boltzmann (relative) entropy and Fisher information obtained by Guillen and Silvestre \cite{GS25}.
\begin{lemma}[Dissipation of entropy along $Q$]\label{lem:H-diss along Q}
    Let $f$ be a probability density that is either local-in-time solution to  \eqref{eq:InfiniteGamma} with initial density $f^\circ \in \mathcal{S}(\R^3)$ or strong solution to~\eqref{eq:infiniteCoulomb} with initial density $f^\circ$ satisfying $f^\circ \in L^1_{2\ell} \cap L^2_\ell \cap L \log L$ for some $\ell\geq 7$ and  $\JSph{f^\circ } < \infty$. The dissipation
    \begin{equation*}
        D_{H, Q}(f) \coloneqq - \big\langle H'(f), Q(f,f) \big\rangle
    \end{equation*}
    is given by
    \begin{equation*}
        D_{H, Q}(f)
        = \frac{1}{2}\int_{(\R^3)^2} |v-v^*|^{2+\gamma} \abs{\Pperp{v-v^*}\nablavwminus \log \ff}^2 \ff \ud{v}\ud{v^*}.
    \end{equation*}
\end{lemma}
\begin{proof}
    By direct computation we have 
    \begin{align*}
        \big\langle H'(f), Q(f,f) \big\rangle &= \int_{\R^3} Q(f,f) \log f\, \ud{v} \\
            &= \int_{(\R^3)^2} \nablavwminus \cdot \left(|v-v^*|^{2+\gamma} \Pperp{v-v^*}\nablavwminus \ff\right) \log f(v)\, \ud{v}\ud{v^*} \\
            &= \int_{(\R^3)^2} \nablavwminus \cdot \left(|v-v^*|^{2+\gamma} \Pperp{v-v^*}\nablavwminus \ff\right) \log f(v^*)\, \ud{v}\ud{v^*} \\
            &= \frac{1}{2} \int_{(\R^3)^2} \nablavwminus  \cdot \left(|v-v^*|^{2+\gamma} \Pperp{v-v^*}\nablavwminus \ff\right) \log \ff\, \ud{v}\ud{v^*}.
    \end{align*}
    Integrating by parts concludes the proof.
\end{proof}
\begin{proposition}[Dissipation of Fisher information along $Q$]\label{prop:I-diss along Q}
    In the setting of \cref{lem:H-diss along Q}, the dissipation
    \begin{equation*}
        D_{I,Q}(f,f) \coloneqq -\big\langle I'(f), Q(f,f) \big\rangle,
    \end{equation*}
    is given, in the variables $(z, z^*) \coloneqq (v-v^*, (v+v^*)/2)$ with $\off(z, z^*) \coloneqq (f \otimes f)(v, v^*)$, by
    \begin{align*}
        D_{I, Q}(f,f)
            &= \frac{1}{8} \int_{({\R^3})^{2}} |z|^{2+\gamma} \abs{\Pperp{z}\gradz\left(\hz \cdot \gradz \log\off\right) + \frac{\gamma}{2|z|}\Pperp{z}\gradz \log\off}^2 \off\,\ud{z}\ud{z^*}\\
            &\quad + \frac{1}{8}\int_{({\R^3})^{2}} |z|^\gamma \left(|z|^2\norm{\left(\Pperp{z}\gradz\right)^2 \log\off}^2_{\mathrm{HS}} - \frac{\gamma^2}{4}\abs{\Pperp{z}\gradz \log\off}^2 \right) \off\,\ud{z}\ud{z^*} \\
            &\quad + \frac{1}{32} \int_{(\R^3)^2} \gradzs \gradz \log\off : \left(\gradzs \gradz \log\off \Az\right) \off\,\ud{z}\ud{z^*}.
    \end{align*}
\end{proposition}
\begin{proof}
    Since the computations mirror the ones of \cref{sec:Dissipation of the weighted Fisher information}, we introduce the analog notation
    \begin{equation*}
        \bbnablavw \coloneqq \left(\nabla_v, \nabla_{v^*}\right), \qquad \notbbnablavw\, \coloneqq \nabla_v - \nabla_{v^*}.
    \end{equation*}
    We have
    \begin{equation*}
        \big\langle I'(f), Q(f,f) \big\rangle = \int_{\R^3} 2\nabla \left(\frac{ Q(f, f)}{f}\right) \cdot \nabla f + \abs{\nabla \log f}^2 Q(f,f) \,\ud{v} \coloneqq \text{T}_1 + \text{T}_2.
    \end{equation*}
    Then
    \begin{align*}
        \text{T}_1 &= 2\int_{(\R^3)^2} \nabla_v \left(\frac{\nablavwminus \cdot \left(A(v-v^*) \nablavwminus \ff\right)}{f \otimes f} \right) \cdot \nabla_v \log f(v) \ff \ud{v}\ud{v^*} \\
            &= 2\int_{(\R^3)^2} \nabla_{v^*} \left(\frac{\nablavwminus \cdot \left(A(v-v^*) \nablavwminus \ff\right)}{f \otimes f} \right) \cdot \nabla_{v^*} \log f(v^*) \ff \ud{v}\ud{v^*} \\
            &= \int_{(\R^3)^2} \left(\bbnablavw \left(\frac{\notbbnablavw \cdot \left(A(v-v^*) \notbbnablavw \ff\right)}{f \otimes f} \right) \cdot \bbnablavw \log \ff\right) \ff \ud{v}\ud{v^*}
    \end{align*}
    and
    \begin{align*}
        \text{T}_2 &= \int_{(\R^3)^2} \abs{\nabla_v \log f(v)}^2 \nablavwminus \cdot \left(A(v-v^*) \nablavwminus \ff\right) \ud{v}\ud{v^*} \\
            &= \int_{(\R^3)^2} \abs{\nabla_{v^*} \log f(v^*)}^2 \nablavwminus \cdot \left(A(v-v^*) \nablavwminus \ff\right) \ud{v}\ud{v^*} \\
            &= \frac{1}{2}\int_{(\R^3)^2} \abs{\bbnablavw \log \ff}^2 \notbbnablavw \cdot \left(A(v-v^*) \notbbnablavw \ff\right) \ud{v}\ud{v^*},
    \end{align*}
    so that
    \begin{align*}
        \big\langle I'(f), Q(f,f) \big\rangle &= \int_{(\R^3)^2} \left(\bbnablavw \left(\frac{\tQ \ff}{f \otimes f} \right) \cdot \bbnablavw \log \ff\right) \ff \\
        &\quad + \frac{1}{2}\int_{(\R^3)^2} \tQ \ff \abs{\bbnablavw \log \ff}^2
    \end{align*}
    with
    \begin{equation*}
        \tQ \ff \coloneqq\: \notbbnablavw \cdot \left(A(v-w) \notbbnablavw \ff\right).
    \end{equation*}
    The proof is now the same as the ones of \cref{lem:two-particle-dissipation} and \cref{prop:coercivity} combined. Note that the contribution at the origin from radial integration also vanishes even though $\gamma = -3$, but now because of symmetry in the first variable; $\Pperp{z}\gradz \log\off(0, z^*) = 0$. The contribution at $+\infty$ is zero again by finiteness of the Fisher information and bounds on the moments.
\end{proof}
\begin{remark}\label{rmk:spherical Fisher}
The standard Fisher information $I(f)$ and the operator $Q(f,f)$ are translation invariant, leading to good dissipation estimates for $I$ along $Q$. On the other hand, the weighted spherical Fisher information $J_1$ breaks the translation invariance and does not interact well with $Q$ by itself. Therefore, we will use the dissipation of $J_1$ along the spherical diffusion
    \begin{equation*}
        L_Sf(v) \coloneqq \nabla \cdot \left(\frac{\Pperp{v}\nabla f(v)}{|v|}\right),
    \end{equation*} to bound the resulting terms. 
\end{remark}
\begin{lemma}\label{lem:dissipation extra inequality}
    Recall~\cref{def:optimal_constants} for $\Lambda_3, \Lambda_3^{\mathrm{sym}}$. Then the following bounds hold:
    \begin{align}
        \int_{\R^3} \left(\frac{\abs{\Pperp{v}\nabla \log f}^4}{|v|^{2}} + 18\frac{\abs{\Pperp{v}\nabla \log f}^2}{|v|^{4}}\right) f \, \ud{v} &\le 9 D_{J_1, L_S}(f) \numberthis\label{eq:norm 4 DJL} \\
        \int_{\R^3} \left(\frac{\abs{\nabla\cdot \lp \Pperp{v}\nabla \log f \rp }^2}{|v|^{2}} + 2\,\frac{\abs{\Pperp{v}\nabla \log f}^2}{|v|^{4}}\right) f\, \ud{v} &\le D_{J_1, L_S}(f) \numberthis\label{eq:norm Delta 2 DJL} \\
        \int_{(\R^3)^2} \frac{\abs{\Pperp{v-v^*}\nablavwminus \log\ff}^4}{|v-v^*|^{-(2+\gamma)}} \ff \,\ud{v}\ud{v^*} &\le 2 \bar C_{\gamma, \Lambda_3^{\mathrm{sym}}} D_{I, Q_\gamma}(f,f) \numberthis\label{eq:norm 4 DIQ} \\
        \int_{(\R^3)^2} \frac{\abs{\nablavwminus\cdot\left(\Pperp{v-v^*}\nablavwminus \log\ff\right)}^2}{|v-v^*|^{-(2+\gamma)}} \ff \,\ud{v}\ud{v^*} &\le \bar C_{\gamma, \Lambda_3^{\mathrm{sym}}} D_{I, Q_\gamma}(f,f) \numberthis\label{eq:norm Delta 2 DIQ}, 
    \end{align} where $\bar C_{\gamma, \Lambda_3^{\mathrm{sym}}} \coloneqq \tfrac{4608}{4 - \gamma^2/\Lambda_3^{\mathrm{sym}}} $.
\end{lemma}
\begin{proof}
    Consider the vector field $X  \coloneqq \abs{\snabla \log \tf}^2\snabla \tf$ on $\S^2$. Then the divergence theorem yields
    \begin{equation*}
        \int_{\S^2} \snabla \cdot X  = \int_{\S^2} \left(3\Delta_{\S^2} \log\tf \abs{\snabla \log\tf}^2 + \abs{\snabla \log\tf}^4\right) \tf = 0,
    \end{equation*}
    By Cauchy-Schwarz inequality, for any $0 < \delta < 1$, we deduce that
    \begin{equation}\label{eq:last step}
        \int_{\S^2} \abs{\snabla \log\tf}^4\,\tf \le \frac{\tfrac{3}{2}\delta^{-1}}{1 - \tfrac{3}{2}\delta} \int_{\S^2} \left(\Delta_{\S^2} \log\tf\right)^2 \tf \le \frac{3\delta^{-1}}{1 - \tfrac{3}{2}\delta} \int_{\S^2} \norm{\snabla^2 \log\tf}_{\mathrm{HS}}^2 \tf,
    \end{equation}
    and $\delta \coloneqq 1/3$ minimizes the constant on the right-hand side of \eqref{eq:last step} to be $18$. Now \eqref{eq:norm 4 DJL} follows from \eqref{eq:dissipation J beta L} recalling \eqref{eq:Gamma 2 HS}. The second estimate follows from the last step in \eqref{eq:last step}. For the third estimate, recall that $\nabla_v - \nabla_{v*} \mapsto 2\nabla_z$ under the change of variables $(z, z^*) \coloneqq (v-v^*, (v+v^*)/2)$, and thus
    \begin{equation*}
        \abs{\Pperp{v-v^*}\nablavwminus \log\ff}^4 = 16\abs{\Pperp{z} \gradz\log\uff}^4.
    \end{equation*}
    Now, since $\uff(r,\omega, z^*) \coloneqq \off(z,z^*) \coloneqq \ff(v,w)$ is symmetric on $\S^2$, seen as the map $\uff(r,\cdot, z^*)$, we obtain the following practical lower bound for the dissipation by combining \cref{prop:I-diss along Q} with the log-Sobolev inequality on $\S^2$ for symmetric functions \cite[Theorem 1.1]{Ji24};
    \begin{align*}
        D_{I, Q_\gamma}(f,f) &\ge \frac{1}{8}\int_{\R^3}\int_{\S^2}\int_0^\infty r^\gamma \left(\norm{\Gamma_2(\log\uff, \log\uff)}^2_{\mathrm{HS}} - \frac{\gamma^2}{4}\abs{\snabla \log\uff}^2 \right) \uff\,\ud{r}\ud{\omega}\ud{z^*} \\
            &\ge \frac{1}{8}\int_{\R^3}\int_{\S^2}\int_0^\infty r^\gamma \left(1 - \frac{\gamma^2}{4 \Lambda_3^{\mathrm{sym}}}\right)\norm{\snabla^2 \log\uff}^2_{\mathrm{HS}} \uff\,\ud{r}\ud{\omega}\ud{z^*}. \numberthis\label{eq:right hand side sym}
    \end{align*}
    Then the constant of the log-Sobolev inequality on $\S^2$ becomes $\Lambda_3^{\mathrm{sym}} \ge 5.5$, and in particular the right-hand side of \eqref{eq:right hand side sym} is non-negative. The symmetry can be readily verified writing $v = \tfrac{z}{2} + z^*$ and $v^* = -\tfrac{z}{2} + z^*$. The proof now follows from \eqref{eq:last step}; $16 \times 18 \times 8 \times 4 = 9216$.
\end{proof}
\subsection{Non-monotonicity of the standard Fisher information: a counterexample} \label{sec:nonMono}

In this section, we prove~\cref{thm:FisherIncrease}, which shows that, in general, the Fisher information is not a Lyapunov functional for the multi-species Landau equation, even for less singular interactions than the Coulomb interaction.

\begin{proof}[Proof of~\cref{thm:FisherIncrease}]
 We recall the two-species system with one species with infinite mass. The equation for the finite-mass species then reads
\begin{align} \label{eq:InfiniteGamma2}
    \partial_t f = Q(f,f) + Lf, \qquad Lf \coloneqq \nabla \cdot \left(|v|^{2+\gamma}\Pperp{v} \nabla f\right).
\end{align}
The local well-posedness of this equation for regular decaying initial data $f^\circ $ follows as~\cref{subsec:local}. We refer to \cite{HHJW25} for an extensive analysis in a more general setting. 

We will construct a Schwartz function $f^\circ_R$ for the initial datum, such that the Fisher information increases near $t=0$. To this end, we compute the time derivative of the Fisher information
\begin{align} \label{heavyIrepr}
    \frac{\d}{\d t} I(f)|_{t=0} &= \big\langle I'(f^\circ) , Lf^\circ\big\rangle  +\big\langle I'(f^\circ) ,Q(f^\circ, f^\circ)\big\rangle.
\end{align}
From~\cref{prop:I prime L}, we have, along the linear operator $L$,
\begin{equation}\label{IfR1:formula}
\begin{aligned}
    \big\langle I'(f^\circ) , Lf^\circ\big\rangle  &= -2\int_{\R^3} |v|^{2+\gamma}\abs{\Pperp{v}\nabla \left(\hv \cdot \nabla \log f^\circ\right) + \frac{\gamma}{2|v|}\Pperp{v}\nabla \log f^\circ}^2 f^\circ\,\ud{v}\\
        &\quad - 2\int_{\R^3} |v|^{2+\gamma} \left(\norm{\left(\Pperp{v}\nabla\right)^2 \log f^\circ}^2_{\mathrm{HS}} - \frac{\gamma^2}{4|v|^2}\abs{\Pperp{v}\nabla \log f^\circ}^2 \right) f^\circ\,\ud{v},
\end{aligned}
\end{equation}
and for the one along the nonlinear collision operator $Q$; under the change of variables $(z, z^*) \coloneqq (v-v^*, (v+v^*)/2)$, set $\off(z, z^*) \coloneqq (f \otimes f)(v, v^*)$, by~\cref{prop:I-diss along Q} we have 
\begin{align*} 
    \big\langle I'(f^\circ) ,Q(f^\circ, f^\circ)\big\rangle 
    &= \frac{1}{8} \int_{({\R^3})^{2}} |z|^{2+\gamma} \abs{\Pperp{z}\gradz\left(\hz \cdot \gradz \log\off^\circ \right) + \frac{\gamma}{2|z|}\Pperp{z}\gradz \log\off^\circ}^2 \off^\circ \d z \d z^* \\
        &\quad + \frac{1}{8}\int_{({\R^3})^{2}} |z|^\gamma \left(|z|^2\norm{\left(\Pperp{z}\gradz\right)^2 \log\off^\circ}^2_{\mathrm{HS}} - \frac{\gamma^2}{4}\abs{\Pperp{z}\gradz \log\off^\circ }^2 \right) \off^\circ \d z \d z^*\\
        &\quad + \frac{1}{32} \int_{(\R^3)^2} \gradzs \gradz \log\off^\circ  : \left(\gradzs \gradz \log\off^\circ \Az\right) \off^\circ \d z \d z^*.
\end{align*}
We define initial data $f^\circ_R$, $R>0$ using the ansatz
\begin{align*}
    f^\circ_R(v) &= f^\circ_{1,R}(v) +  f^\circ_{2,R}(v),  
\end{align*}
with
\begin{align*}
f^\circ_{R,1}(v) = \exp \left(R |v|^{\frac{|\gamma|}2} e_1 \cdot \hat{v}\right) \chi \lp \frac{v}{r}\rp \qquad f^\circ_{R,2}(v) =\kappa \chi \lp \frac{v}{\tau}\rp,
\end{align*}
where $\chi$ is a radial cutoff function supported on the annulus defined by $\frac12 \leq |x| \leq 2$ and 
\begin{align} \label{rRscaling}
    r= \frac{R^{2/\gamma}}{|\log (2+R)|} .
\end{align} For $R\geq  1$ we choose  $\kappa(R),\tau(R)>0$ such that the function $f^\circ_R$ has unit mass and kinetic energy. 

\medskip 

\noindent
\textbf{Step 1.} We first observe  that $\kappa(R) \rightarrow \kappa_\infty>0$, $\tau(R) \rightarrow \tau_\infty>0$ converge to finite positive values as $R\rightarrow \infty$ . Hence, for  $R$ sufficiently large,  $f^\circ_{R,1}$ and $f^\circ_{R,2}$ have disjoint support and we can compute the dissipations along $L$ separately;
\begin{align*} 
    \big\langle I'(f^\circ_R) , Lf^\circ_R \big\rangle &= \big\langle I'(f^\circ_{R,1}), Lf^\circ_{R,1}\big\rangle + \big\langle I'(f^\circ_{R,2}) , Lf^\circ_{R,2}\big\rangle.
\end{align*}
Using again the convergence of $\kappa(R)$ and $\tau(R)$, we conclude
\begin{align*}
    \limsup_{R\rightarrow \infty}|\big\langle I'(f^\circ_{R,2}) , Lf^\circ_{R,2}\big\rangle| < \infty.
\end{align*}
\textbf{Step 2.}  We have the following identities:
\begin{align*}
    \Pperp{v} \nabla \log f^\circ_{R,1}(v) &= R |v|^{-(1+\gamma/2)} \Pperp{v} e_1 \\
    \Pperp{v} \nabla \left( \hv \cdot \log f^\circ_{R,1}(v)\right) &= R \frac{|\gamma|}{2}|v|^{-(2+\gamma/2)} \Pperp{v} e_1 \\
     \left(\Pperp{v} \nabla\right)^2 \log f^\circ_{R,1}(v) &= -R|v|^{-(2+\gamma/2)} \left(\left(e_1 \cdot \hv\right)\Pperp{v} + \left(\Pperp{v}e_1\right) \otimes \hv\right).
\end{align*}
The functions $f^\circ_{R,1}$ are constructed such that the first term in~\eqref{IfR1:formula} vanishes, and therefore
\begin{align*}
    \big\langle I'(f^\circ_{R,1}), Lf^\circ_{R,1}\big\rangle =- 2\int_{\R^3} |v|^{2+\gamma} \left(\norm{\left(\Pperp{v}\nabla\right)^2 \log f^\circ_{R,1}}^2_{\mathrm{HS}} - \frac{\gamma^2}{4|v|^2}\abs{\Pperp{v}\nabla \log f^\circ_{R,1}}^2 \right) f^\circ_{R,1}\, \d v.
\end{align*}
Computing the derivatives, we find that
\begin{align*}
    \big\langle I'(f^\circ_{R,1}), Lf^\circ_{R,1}\big\rangle &=- 2R^2\int_{\R^3} |v|^{-2} \left(\| \Pi_v^\perp (\hat v\cdot e_1) + \Pi_v^\perp e_1 \otimes \hat v\|^2_{\mathrm{HS}} - \frac{\gamma^2}{4} |\Pi^\perp_{v} e_1|^2  \right) f^\circ_{R,1}\, \d v \\
    &=- 2R^2\int_0^\infty |l|^{-2} \int_{\partial B_l} \lp 2 (\hat v\cdot e_1)^2 - \lp \frac{\gamma^2-4}{4} \rp |\Pi_{\hat v}^\perp e_1|^2 \rp f^\circ_{R,1}(l \hat v)\, \d{\hat v} \d l.
\end{align*}
Using $\gamma^2>8$, we obtain that 
\begin{align*}
    \int_{\partial B_1} \lp 2 (\hat v\cdot e_1)^2 - \lp \frac{\gamma^2-4}{4} \rp |\Pi_{\hat v}^\perp e_1|^2 \rp \d{\hat v} < 0.
\end{align*}
Hence, for  $R>0$ large enough we obtain for some $c_\gamma>0$ that 
\begin{align*}
    |\big\langle I'(f^\circ_{R,1}), Lf^\circ_{R,1}\big\rangle &\geq   c_{\gamma} R^2r .
\end{align*}
\textbf{Step 3.}  We claim that
\begin{align*}
    \big\langle I'(f^\circ_R) , Q(f^\circ_R, f^\circ_R)\big\rangle| \leq C \big(1  + R^2 r^{|\gamma|-1}\big) .
\end{align*}
where $r$ depends on $R$ as in~\eqref{rRscaling}.
 
To this end, we first observe that
\begin{align*}
    |\nabla_v \log f^\circ_{R,1}| \leq {\bf 1}_{|v|\leq 2r}
        C \left(R|v|^{\frac{|\gamma|}{2} -1} + r^{-1}\right), \qquad \norm{\nabla_v^2 \log f^\circ_{R,1}}_{\mathrm{HS}} \leq {\bf 1}_{|v|\leq 2r}
        C \left(R|v|^{\frac{|\gamma|}{2} -2} + r^{-2}\right).
\end{align*}
Inserting this into the formula for $\big\langle I'(f^\circ_R) , Q(f^\circ_R, f^\circ_R)\big\rangle$, the claim follows by straightforward integration.
\medskip

\noindent
\textbf{Step 4.}  We collect the estimates from \textbf{Steps 1-3} and insert them into~\eqref{heavyIrepr}. Courtesy of the scaling relation~\eqref{rRscaling}, we obtain for $R>1$ large enough
\begin{align*}
    \frac{\d}{\d t} I(f)|_{t=0} > c_\gamma R^2r - C \lp 1+R^2 r^{|\gamma|-1} \rp > 0,
\end{align*}
for $R>0$ sufficiently large, establishing non-monotonicity.

The monotonicity of the Fisher information $I(f)$ for strong solutions with even initial data follows quickly from the result of Guillen and Silvestre: Since $f(t,v)=f(t,-v)$ holds for local-in time solution, the Fisher information is dissipated by $L$. To see this, we recall~\cref{prop:I prime L} 
\begin{align*}
    \ \big\langle I'(f) , Lf\big\rangle = &-2\int_{\R^3} |v|^{2+\gamma}\abs{\Pperp{v}\nabla \left(\hv \cdot \nabla \log f\right) + \frac{\gamma}{2|v|}\Pperp{v}\nabla \log f}^2 f\,\ud{v}\\
        &-2\int_{\R^3} |v|^{2+\gamma} \left(\norm{\left(\Pperp{v}\nabla\right)^2 \log f}_{\mathrm{HS}}^2 - \frac{\gamma^2}{4|v|^2}\abs{\Pperp{v}\nabla \log f}^2 \right) f\,\ud{v} \leq 0,  
\end{align*}
where we use that~\eqref{eq:logSob} holds with $\Lambda_3^{\mathrm{sym}} \geq 5.5$ for even functions.
\end{proof}

\subsection{A new Lyapunov functional} \label{sec:Thm3proof}
In this section, we prove~\cref{thm:Lambda dissipation}, which shows that $\Lambda$ defined in \eqref{eq:new_Lyapunov} is a Lyapunov functional for the multi-species Landau equation for singular interactions up to Coulomb.
\begin{proof}[Proof of~\cref{thm:Lambda dissipation}]
As mentioned in \cref{rmk:spherical Fisher}, we estimate the derivative of $J_1$ along  $Q_\gamma$ in terms of the other dissipations, courtesy of \cref{lem:dissipation extra inequality}. It holds
\begin{align*}
    \big\langle J'_1(f), &\,Q_\gamma(f, f)\big\rangle = \int_{\R^3} \frac{\abs{\Pperp{v}\nabla_v \log f}^2}{|v|}\, Q_\gamma(f,f) + 2\, \frac{\Pperp{v}\nabla_v \log f(v)}{|v|} \cdot \Pperp{v}\nabla_v \left(f^{-1} Q_\gamma(f,f)\right) f \, \ud{v} \\
        = &-\int_{\R^3} \frac{\abs{\Pperp{v}\nabla_v \log f}^2}{|v|}\, Q_\gamma(f,f)\, \ud{v} \\
        & -2 \int_{(\R^3)^2} \nabla_v \cdot \left(\frac{\Pperp{v}\nabla_v \log f(v)}{|v|}\right) \frac{\nablavwminus \cdot \left(\Pperp{v-v^*} \nablavwminus  \ff\right)}{|v-v^*|^{-(2 + \gamma)}}\  \ud{v}\ud{v^*} \\
        = &-\int_{(\R^3)^2} \frac{\abs{\Pperp{v}\nabla_v \log f(v)}^2}{|v|} \frac{\nablavwminus \cdot \left(\Pperp{v-v^*} \nablavwminus \ff\right)}{|v-v^*|^{-(2 + \gamma)}} \ud{v}\ud{v^*} \\
        & -2 \int_{(\R^3)^2} \nabla_v \cdot \left(\frac{\Pperp{v}\nabla_v \log f(v)}{|v|}\right) \frac{\nablavwminus \cdot \left(\Pperp{v-v^*} \nablavwminus \ff\right)}{|v-v^*|^{-(2 + \gamma)}} \ud{v}\ud{v^*},
    \end{align*}
where we integrate by parts in the second equality. We recall the relation
\begin{align*}
    \frac{\nablavwminus \cdot \left(\Pperp{v-v^*} \nablavwminus \ff\right)}{f \otimes f} = &\nablavwminus \cdot \left(\Pperp{v-v^*} \nablavwminus \log \ff\right) \\
        &+ \abs{\Pperp{v-v^*} \nablavwminus \log \ff}^2.
\end{align*}
Hölder and triangle inequalities then yield
\begin{align*}
    \big\langle J'_1(f), \,Q_\gamma(f, f)\big\rangle &\le \left( \int_{\R^3} \frac{\abs{\Pperp{v}\nabla_v \log f(v)}^4}{|v|^2} \left(\int_{\R^3} \frac{f(v^*)}{|v-v^*|^{-(2 + \gamma)}} \ud{v^*}\right) f(v) \,\ud{v}\right)^{1/2} \\
        & \quad \times \left[\left(\int_{(\R^3)^2} \frac{\abs{\nablavwminus \cdot \left(\Pperp{v-v^*} \nablavwminus \log \ff\right)}^2}{|v-v^*|^{-(2 + \gamma)}} \ff \ud{v}\ud{v^*}\right)^{1/2} \right. \\
        &\qquad + \left.\left(\int_{(\R^3)^2} \frac{\abs{\Pperp{v-v^*} \nablavwminus \log \ff}^4}{|v-v^*|^{-(2 + \gamma)}} \ff \ud{v}\ud{v^*}\right)^{1/2}\right] \\
        & \quad + 2\left( \int_{\R^3} \frac{\abs{\nabla_v \cdot \left(\Pperp{v}\nabla_v \log f(v)\right)}^2}{|v|^2} \left(\int_{\R^3} \frac{f(v^*)}{|v-v^*|^{-(2 + \gamma)}} \ud{v^*}\right) f(v) \,\ud{v}\right)^{1/2} \\
        & \quad \times \left[\left(\int_{(\R^3)^2} \frac{\abs{\nablavwminus \cdot \left(\Pperp{v-v^*} \nablavwminus \log \ff\right)}^2}{|v-v^*|^{-(2 + \gamma)}} \ff \ud{v}\ud{v^*}\right)^{1/2} \right. \\
        &\qquad + \left.\left(\int_{(\R^3)^2} \frac{\abs{\Pperp{v-v^*} \nablavwminus \log \ff}^4}{|v-v^*|^{-(2 + \gamma)}} \ff \ud{v}\ud{v^*}\right)^{1/2}\right] \\
        &\le C_{\gamma,\Lambda_3^{\mathrm{sym}}} \norm{\int_{\R^3} |\cdot - \,v^*|^{2+\gamma} f(v^*) \, \ud{v^*}}_{L^\infty}^{1/2} D^{1/2}_{J_1, L_S}(f) D^{1/2}_{I, Q_\gamma}(f), \numberthis \label{eq:before interpolation sobolev}
\end{align*}
where in the last inequality we used \eqref{eq:norm 4 DJL}, \eqref{eq:norm Delta 2 DJL}, and \eqref{eq:norm Delta 2 DIQ} from \cref{lem:dissipation extra inequality}. The constant is given by
\begin{equation*}
    C_{\gamma,\Lambda_3^{\mathrm{sym}}} \coloneqq 5 \lp 1 + \sqrt{2}\rp \sqrt{ \bar C_{\gamma, \Lambda_3^{\mathrm{sym}}}}.
\end{equation*}
Recall that $-3 \le \gamma < -2$. Therefore, by Young's inequality, splitting $|\cdot - v^*|$ into two regions, one close and one away from the origin, we get
\begin{equation}\label{eq:convolution l infty}
    \norm{\int_{\R^3} \frac{f(v^*)}{|\cdot - \,v^*|^{-(2+\gamma)} }\, \ud{v^*}}_{L^\infty} \le 1 + \norm{|\cdot|^{2+\gamma} \1_{|\cdot| < 1}}_{L^{-(3-\delta)/(2 + \gamma)}} \norm{f}_{L^{(3-\delta)/(5+\gamma-\delta)}} = 1 + C_{\gamma,\delta}\norm{f}_{L^{(3-\delta)/(5+\gamma-\delta)}}
\end{equation}
for any $\delta > 0$ small enough ($0 < \delta < 5+\gamma$) and
\begin{equation*}
    C_{\gamma,\delta} \coloneqq \left(\frac{4\pi}{\delta}\right)^{-\frac{(2+\gamma)}{(3-\delta)}},
\end{equation*}
where the choice of exponents stems from the criticality $|\cdot|^{-1} \1_{|\cdot| \le 1} \notin L^3$. Indeed, the exponent of the Lebesgue norm of $f$ should be as small as possible to control the norm itself by Fisher information with a small exponent, through the Sobolev inequality \cite{CFMP09} and the $L^1 \cap L^3$ interpolation. Let $C_S$ be the optimal constant of the Sobolev inequality, then
\begin{equation*}
    \norm{f}_{L^{(3-\delta)/(5+\gamma-\delta)}} \le \norm{f}_{L^3}^{2\theta_{\gamma,\delta}} = \norm{\sqrt{f}}^{4\theta_{\gamma,\delta}}_{L^{6}} \le \left(C_S^2\norm{\sqrt{f}}^2_{\dot{H}^{1}}\right)^{2\theta_{\gamma,\delta}} = \left(\frac{C_S^2}{4} I(f)\right)^{2\theta_{\gamma,\delta}}, \:\: \theta_{\gamma,\delta} \coloneqq -\frac{3(2+\gamma)}{4(3-\delta)}.
\end{equation*}
Recalling \eqref{eq:before interpolation sobolev}, we obtain
\begin{equation*}
    \big\langle J'_1(f), \,Q_\gamma(f, f)\big\rangle \le C_{S,\gamma,\Lambda_3^{\mathrm{sym}}} \left[\left(1 +I(f)\right)^{2\theta_{\gamma,\delta}}D_{J_1, L_S}(f) D_{I, Q_\gamma}(f)\right]^{1/2}
\end{equation*}
with
\begin{equation*}
    C_{S,\gamma,\Lambda_3^{\mathrm{sym}}, \delta} \coloneqq C_{\gamma,\Lambda_3^{\mathrm{sym}}} \times \left(1 + C_{\gamma,\delta}\left(\tfrac{C_S^2}{4}\right)^{2\theta_{\gamma,\delta}}\right)^{{\frac{1}{2}}}.
\end{equation*}
Note that $\Lambda(f)$, defined in~\eqref{eq:new_Lyapunov} is non-negative due to the Boltzmann relative entropy. We thus have the upper bound $I(f) \le \Lambda(f)$ and
\begin{align*}
    \big\langle J'_1(f), \,Q_\gamma(f, f)\big\rangle &\le C_{S,\gamma,\Lambda_3^{\mathrm{sym}},\delta}  \left[\left(1 + \Lambda(f)\right)^{2 \theta_{\gamma,\delta}}D_{J_1, L}(f) D_{I, Q_\gamma}(f, f)\right]^{1/2} \\
        &\le C_{S,\gamma,\Lambda_3^{\mathrm{sym}},\delta}\left(1 + \Lambda(f)\right)^{\theta_{\gamma,\delta}} \left(\frac{\kappa}{2}D_{J_1, L}(f) + \frac{1}{2\kappa}D_{I, Q_\gamma}(f, f) \right). \numberthis\label{eq:upper bound J prime Q}
\end{align*}
Moreover, by \cref{prop:I prime L} for the singularity $|\cdot|^{-1}$ combined with the log-Sobolev inequality, we obtain the following bound by interpolation;
\begin{align*}
    \big\langle I'(f) , Lf\big\rangle &\le 2\left(\frac{9}{4} - 2\right)\int_{\R^3} \frac{\abs{\Pperp{v}\nabla \log f}^2}{\abs{v}^3} f \\
        &\le \frac{1}{2} \left(\int_{\R^3} \frac{\abs{\Pperp{v}\nabla \log f}^2}{\abs{v}} f\right)^{1/3}\left(\int_{\R^3} \frac{\abs{\Pperp{v}\nabla \log f}^2}{\abs{v}^4} f\right)^{2/3}.
\end{align*}
Hence, recalling \cref{lem:H-diss along L} and \eqref{eq:norm 4 DJL}, we get by Young's inequality for any $\eta > 0$,
\begin{equation}\label{eq:bound I prime L in terms of dissipations of H and Q}
    \big\langle I'(f) , L_Sf\big\rangle \le 2^{-5/3} \left(\eta^{-2} D_{H, L_S}(f)\right)^{1/3} \left(\eta\,D_{J_1, L_S}(f)\right)^{2/3} \le 2^{-5/3} \left(\frac{1}{3\eta^{2}} D_{H, L_S}(f) + \frac{2\eta}{3} D_{J_1, L_S}(f)\right).
\end{equation}
Now, we can control the evolution of the new functional $\Lambda$. For this, we assume that
\begin{equation}\label{eq:assume Lambda}
    1 + \Lambda(f) \le C^\circ R
\end{equation} for some constant $C^\circ > 0$ large enough; note that this holds initially at time $t=0$ with 
\begin{equation*}
    C^\circ \coloneqq 1 + I(f^\circ) + J(f^\circ) + \HRel{f^\circ}{\mathcal{M}_{T^\circ,1}}
\end{equation*}
if we furthermore assume that $a \le 1$ and $R \ge 1$. Then, its time derivative is given by
\begin{align*}
    \frac{\d}{\d t} &\Lambda(f) \\ &= -D_{I,Q_\gamma}(f) + c_{ei}\big\langle I', L_S(f)\big\rangle + a\big\langle J'_1, Q_\gamma(f, f)\big\rangle - ac_{ei} D_{J_1, L_S}(f) - RD_{H, Q_\gamma}(f, f) - Rc_{ei}D_{H, L_S}(f) \\
        &\le -D_{I,Q_\gamma}(f) + \frac{c_{ei}}{2^{5/3}} \left(\frac{1}{3\eta^{2}} D_{H, L_S}(f) + \frac{2\eta}{3} D_{J_1, L_S}(f)\right) + \frac{ aC_{S,\gamma,\Lambda_3^{\mathrm{sym}},\delta}\left(C^\circ R\right)^{\theta_{\gamma,\delta}}\kappa}{2} D_{J_1, L_S}(f) \\
        &\quad + \frac{ a\,C_{S,\gamma,\Lambda_3^{\mathrm{sym}},\delta}\left(C^\circ R\right)^{\theta_{\gamma,\delta}}}{2\kappa}D_{I, Q_\gamma}(f,f) 
        - ac_{ei} D_{J_1, L_S}(f) - RD_{H, Q_\gamma}(f,f) - Rc_{ei}D_{H, L_S}(f) \\
        &= \,\left(\frac{a\,C_{S,\gamma,\Lambda_3^{\mathrm{sym}},\delta}\left(C^\circ R\right)^{\theta_{\gamma,\delta}}}{2\kappa} - 1\right) D_{I, Q_\gamma}(f,f) \numberthis\label{eq:dissipation coef I Q}\\
        &\quad + \left(\frac{c_{ei}}{2^{2/3}\,3}\eta +  \frac{ a\,C_{S,\gamma,\Lambda_3^{\mathrm{sym}},\delta}\left(C^\circ R\right)^{\theta_{\gamma,\delta}}\,\kappa}{2} - a c_{ei} \right)D_{J_1, L_S}(f) \numberthis\label{eq:dissipation coef J L}\\
        &\quad + \left(\frac{c_{ei}}{2^{5/3}\,3}\,\eta^{-2} - Rc_{ei}\right)D_{H, L_S}(f) - RD_{H, Q_\gamma}(f,f), \numberthis\label{eq:dissipation coef H Q}
\end{align*}
where in the inequality we used the upper bounds \eqref{eq:upper bound J prime Q} and \eqref{eq:bound I prime L in terms of dissipations of H and Q}.

\bigskip 

\noindent
\textbf{Case 1: $-3 < \gamma < -2$.} We choose the parameters $\kappa$ and $\eta$ such that all prefactors are non-positive. For this, we first set
\begin{equation*}
    \kappa \coloneqq a\,C_{S,\gamma,\Lambda_3^{\mathrm{sym}},\delta} \left(C^\circ R\right)^{\theta_{\gamma,\delta}}, \quad \text{hence} \quad \frac{1}{2} = \frac{a\,C_{S,\gamma,\Lambda_3^{\mathrm{sym}},\delta}\left(C^\circ R\right)^{\theta_{\gamma,\delta}}}{2\kappa} 
\end{equation*}
so that the coefficient \eqref{eq:dissipation coef I Q} of $D_{I,Q_\gamma}(f,f)$ becomes non-positive. Second, we set
\begin{equation*}
    \eta \coloneqq \frac{1}{\sqrt{2^{2/3}3 R}}, \quad \text{hence} \quad \frac{c_{ei}}{2^{5/3}\,3}\, \eta^{-2} = \frac{Rc_{ei}}{2}
\end{equation*}
so that the coefficient \eqref{eq:dissipation coef H Q} of $D_{H,Q_\gamma}(f,f)$ becomes $-R/2$. Third, we need to ensure that the coefficient of $D_{J_1,L_S}(f)$ is non-positive; substituting $\kappa$ and $\eta$ in \eqref{eq:dissipation coef J L}, this amounts to
\begin{equation*}
    \frac{c_{ei}}{6\sqrt{3}} \frac{1}{\sqrt{R}} + \frac{a^2\, C^2_{S,\gamma,\Lambda_3^{\mathrm{sym}},\delta} \left(C^\circ R\right)^{2\theta_{\gamma,\delta}}}{2} - ac_{ei} \le 0.
\end{equation*}
We can absorb the first term setting $\frac{a}{4} \coloneqq \tfrac{1}{6\sqrt{3}}\tfrac{1}{\sqrt{R}}$ so that it remains to verify
\begin{equation}\label{eq:to verify}
    a\, C^2_{S,\gamma,\Lambda_3^{\mathrm{sym}},\delta} \left(C^\circ R\right)^{2\theta_{\gamma,\delta}} \le \frac{c_{ei}}{2}.
\end{equation}
In the case where $-3 < \gamma < -2$, we can choose $0 <\delta < 3(3+\gamma)$ for which $\theta_{\gamma,\delta} < 1/4$, hence \eqref{eq:to verify} holds for $R > 0$ large enough. This choice satisfies in particular $\delta < 5 + \gamma$ and the $L^\infty$-estimate \eqref{eq:convolution l infty} is thus valid. Then
\begin{equation*}
    \frac{\d}{\d t} \Lambda(f) \le -\frac{1}{2}D_{I,Q_\gamma}(f,f) - \frac{ac_{ei}}{2}D_{I,L_S}(f) - \frac{R}{2}D_{H, Q_\gamma}(f,f) - RD_{H, L_S}(f) \le 0.
\end{equation*}
Recall that we assumed a priori \eqref{eq:assume Lambda}, namely that $1 + \Lambda(f) \le C^\circ R$, which holds at time $t=0$. This holds a posteriori for all time since $\Lambda$ is non-decreasing given $1 + \Lambda^\circ(f) \le C^\circ R$, which concludes the proof in that case.
\bigskip

\noindent
{\bf Case 2: $\gamma = -3$.} For $\lambda>0$, consider the rescaled function $f_\lambda $ given by
\begin{align*}
    f_\lambda(v) = \lambda^{-3} f(v/\lambda).
\end{align*}
Pick $\lambda>0$ such that $I(f^\circ_\lambda)=1$. By scale-invariance of the equation, it suffices to prove the claim for $f_\lambda$, and we consider
\begin{equation} \label{eq:Ccirclambda}
    C^\circ_\lambda  \coloneqq 1 + I(f^\circ_\lambda) + J(f^\circ_\lambda ) + \HRel{f^\circ_\lambda }{\mathcal{M}_{T^\circ/\lambda^2,1}}.
\end{equation}

Further, we remark that for $\gamma=-3$ we have  $\theta_{-3,\delta} = \tfrac{1}{4}\tfrac{3}{3-\delta}$. Now we pick $R=1$ and fix the other parameters such that the prefactors in~\eqref{eq:dissipation coef I Q} and~\eqref{eq:dissipation coef H Q} vanish; that is, 
\begin{equation*}
    \kappa \coloneqq \frac{a\,C_{S,-3,\Lambda_3^{\mathrm{sym}},\delta} \left(C^\circ_\lambda \right)^{\theta_{\gamma,\delta}}}{2}, \quad \eta \coloneqq \frac{1}{\sqrt{2^{5/3}3}},
\end{equation*}
and it remains to determine constants such that 
\begin{equation*}
    \frac{c_{ei}}{12\sqrt{6}} + \frac{a^2\, C^2_{S,-3,\Lambda_3^{\mathrm{sym}},\delta} \left(C^\circ_\lambda  \right)^{2\theta_{\gamma,\delta}}}{4} - ac_{ei} \le 0.
\end{equation*}
We set $a \coloneqq \tfrac{1}{6\sqrt{6}}$ and the last inequality holds as soon as
\begin{equation} \label{eq:cei_bound}
    \frac{1}{9\sqrt{6}} \left(C^\circ_\lambda \right)^{3/(6-2\delta)} C^2_{S,-3,\Lambda_3^{\mathrm{sym}},\delta}  \le c_{ei},
\end{equation}
Hence, we obtain monotonicity of the functional $\Lambda(f)$ if the inequality above holds for some $0 < \delta < 2$.  We finally obtain the well-posedness criterion~\eqref{eq:Coulombcondition} by picking $\delta=\frac32$, and  recasting the condition on $f^\circ_\lambda$ in terms of $f^\circ$.
\end{proof}
\begin{remark} 
We can obtain an explicit value for the constant $K$ in~\eqref{eq:Coulombcondition}, by inserting $\delta=\frac32$ in~\eqref{eq:cei_bound} which bounds $K \leq  3 \times 10^5 $.
\end{remark}

\section*{Acknowledgements}
J. Junné is supported by Dutch Research Council (NWO) \href{https://www.nwo.nl/en/projects/ocenwm20251}{\includegraphics[height=\fontcharht\font`\B]{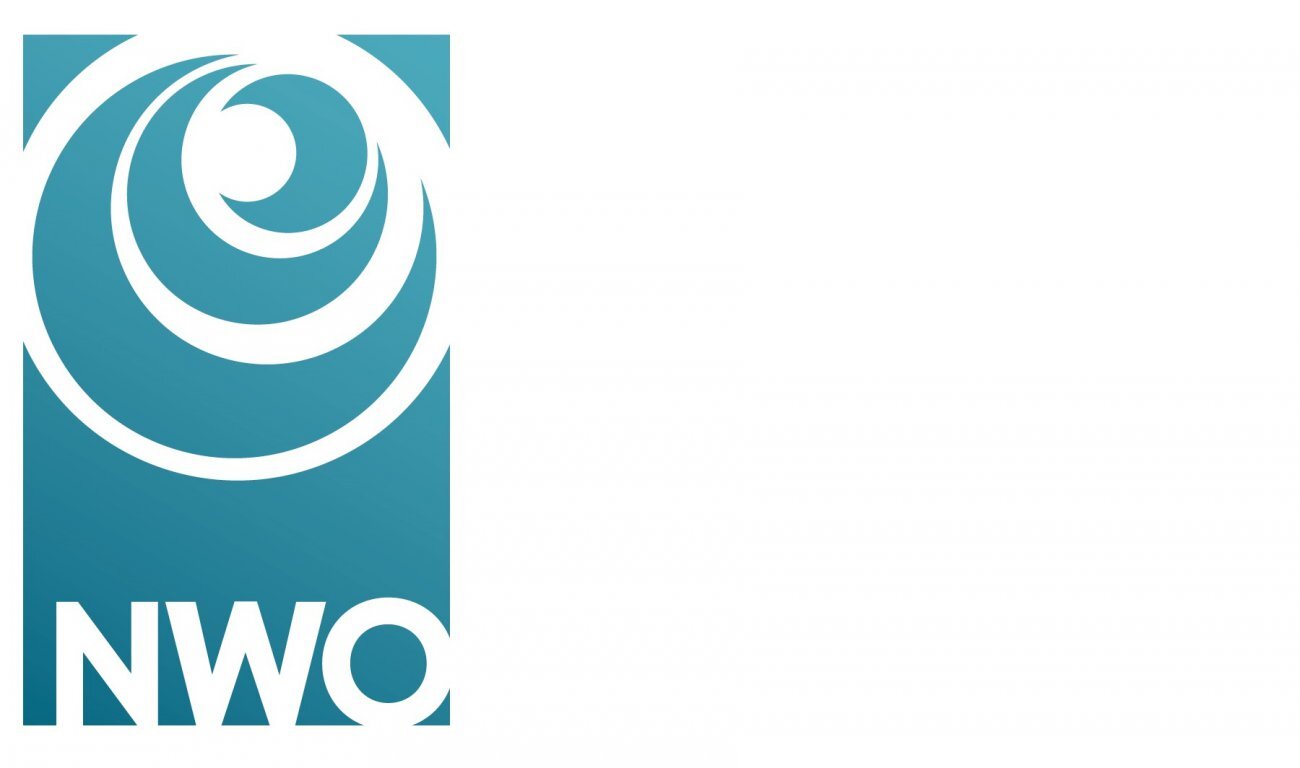}} \hspace{-10pt} under the Open Competitie ENW project Interacting particle systems and Riemannian geometry -- OCENW.M20.251. H. Yolda\c{s} is supported by the Dutch Research Council (NWO) \href{https://www.nwo.nl/en/projects/viveni222288}{\includegraphics[height=\fontcharht\font`\B]{image.jpeg}} \hspace{-10pt} under the NWO-Talent Programme Veni ENW project MetaMathBio -- VI.Veni.222.288. R. Winter and H. Yolda\c{s} would like to thank the Isaac Newton Institute for Mathematical Sciences for the support and hospitality during the INI Retreats programme when work on this paper was undertaken. This work was supported by EPSRC Grant Number EP/Z000580/1. 
For the purpose of open access, the authors have applied a Creative Commons Attribution (CC-BY) licence to any Author Accepted Manuscript version arising from this submission.

\newpage 

\appendix
\section{Spherical calculus}\label{sec:appendixA}
In this appendix, we gather some spherical calculus identities that will be used throughout the manuscript.

We decompose the gradient as
\begin{equation*}
    \nabla_z = \left(\hz \otimes \hz\right)\nabla_z + \Pperp{z} \nabla_z = \hz \rpartial + \frac{1}{|z|}\hznabla,
    \quad \hz \coloneqq \frac{z}{|z|},
    \quad \rpartial \coloneq \hz \cdot \nabla_z, 
    \quad \hznabla \coloneq |z|\Pperp{z} \nabla_z. 
\end{equation*}
If we further perform the change of variables $z \mapsto r\omega$ with $(r, \omega) \coloneqq (|z|, \hz)$, then for scalar functions we have
\begin{equation}\label{eq:z to spherical on scalars}
    \hznabla \mapsto \snabla, 
    \quad \rpartial \mapsto \partial_r,
    \quad \gradz \mapsto \omega \partial_r + \frac{1}{r}\snabla, 
\end{equation}
We denote by $\tests(r,\omega) \coloneqq g(z)$ a smooth function on $\R^3$ in spherical coordinates. The following handy relations hold;
\begin{equation*}
    \rpartial \hz = 0, \quad \hznabla |z| = 0, \quad \hznabla \hz = \Pperp{z}, \quad \rpartial \hznabla\, \testf = \hznabla \rpartial\, \testf.
\end{equation*}
Most of the relations in the manuscript contain $\log$; we write the sequel of the appendix that way, but these relations are of course general.
\begin{lemma}[Spherical Hessian and Laplacian identities]\label{lem:spherical Hess and Delta}
    \begin{equation}\label{eq:spherical hessian formula}
        \frac{1}{r^2}\snabla^2 \log\tests = \Pperp{z} \nabla^2_z \log\testf\, \Pperp{z} - \frac{1}{|z|}\rpartial \log\testf\, \Pperp{z}.
    \end{equation}
    In particular,
    \begin{equation}\label{eq:nabla nabla omega traced pi is spherical Laplacian}
       \frac{1}{r^2} \Delta_{\S^2} \log\tests = \Pperp{z} : \left(\Pperp{z} \nabla^2_z \log\testf\, \Pperp{z} - \frac{1}{|z|}\rpartial \log\testf\, \Pperp{z}\right).
    \end{equation}
\end{lemma}
\begin{proof}
    We have (see e.g. \cite{AMT13}) that
    \begin{equation*}
        \left(\frac{1}{r^2}\snabla^2 \log\tests\right) \sigma = \left(\Pperp{z} \nabla^2_z \log\testf\right) \sigma + \Pperp{z} D_\sigma \Pperp{z} \nabla_z \log\testf, \quad \sigma \in T_{\hz}(\mathbb{S}^2) \subset \R^3.
    \end{equation*}
    Now, since $\Pperp{z} \sigma = \sigma$,
    \begin{equation*}
        \left(\Pperp{z} \nabla^2_z \log\testf\right) \sigma = \left(\Pperp{z} \nabla^2_z \log\testf\, \Pperp{z}\right) \sigma,
    \end{equation*}
    and, using $D_\sigma \hz = (\nabla_z \hz)\sigma = \tfrac{1}{r}\Pperp{z} \sigma = \tfrac{1}{r}\sigma$, we obtain
    \begin{equation*}
        \Pperp{z} D_\sigma \Pperp{z} \nabla_z \log\testf 
        = -\frac{1}{|z|}\Pperp{z}\left(\sigma \otimes \hz + \hz \otimes \sigma\right) \nabla_z \log\testf 
        = - \frac{1}{|z|} \left(\sigma \otimes \hz\right) \nabla_z \log\testf
        = -\frac{1}{|z|}\rpartial \log\testf\, \Pperp{z} \sigma.
    \end{equation*}
    \eqref{eq:nabla nabla omega traced pi is spherical Laplacian} follows noticing that $\Pperp{z} = \sigma_1 \otimes \sigma_1  + \sigma_2 \otimes \sigma_2$ for any orthonormal basis $\{\sigma_1, \sigma_2\}$ of $T_{\hz}\S^2 \subset \R^3$.
\end{proof}
\begin{lemma}[Decomposition of the Euclidean Hessian]\label{lem:decomposition Euclidean Hessian}
    \begin{align*}
    \nabla_z^2 \log\testf &= \rpartial^2 \log\testf\, \hz \otimes \hz + \frac{2}{|z|} \left(\hz \otimes_{\mathrm{sym}}\! \hznabla \rpartial \log\testf\right) - \frac{2}{|z|^2} \left(\hz \otimes_{\mathrm{sym}} \!\hznabla \log\testf\right)
        \\ &\quad + \left(\Pperp{z} \nabla^2_z \log\testf\, \Pperp{z} - \frac{1}{|z|}\rpartial \log\testf\, \Pperp{z}\right) + \frac{1}{|z|}\rpartial \log\testf\, \Pperp{z}.
    \end{align*}
\begin{proof}
    We decompose the Euclidean Hessian as
    \begin{equation*}
        \nabla^2_z \log\testf = \nabla^2_z \log\testf \left(\hz \otimes \hz\right) + \nabla^2_z \log\testf\, \Pperp{z}.
    \end{equation*}
    One checks that
    \begin{equation*}
        \nabla^2_z \log\testf \left(\hz \otimes \hz\right) = \rpartial^2 \log\testf\, \hz \otimes \hz + \frac{1}{|z|} \hznabla \rpartial \log\testf \otimes \hz - \frac{1}{r^2} \hznabla \log\testf \otimes \hz.
    \end{equation*}
    Similarly,
    \begin{equation*}
        \nabla^2_z \log\testf\, \Pperp{z} = \left( \hz \otimes \hz\right)\nabla^2_z \log\testf\, \Pperp{z} + \Pperp{z} \nabla^2_z \log\testf\, \Pperp{z},
    \end{equation*}
    where
    \begin{equation*}
        \left( \hz \otimes \hz\right)\nabla^2_z \log\testf\, \Pperp{z} = \frac{1}{|z|} \hz \otimes \hznabla \rpartial \log\testf - \frac{1}{r^2} \hz \otimes \hznabla \log\testf.
    \end{equation*}
    We conclude recalling \eqref{eq:spherical hessian formula}.
\end{proof}
\end{lemma}
\begin{lemma}\label{lem:decomposition Hessian-ish}
    \begin{align*}
    \nabla_z\left(|z|^{2+\gamma}\Pperp{z}\nabla_z \log\testf\right) &= -2|z|^{2+\gamma} \left(\hz \otimes_{\mathrm{sym}}\! \hznabla \log\testf\right) + \left(2 + \gamma\right) |z|^\gamma \hz \otimes \hznabla \log\testf \\
        &\quad + |z|^{1+\gamma}\hz \otimes \hznabla\,\rpartial \log\testf + |z|^{2+\gamma} \left( \Pperp{z} \nabla_z^2 \log\testf\, \Pperp{z} - \frac{1}{|z|} \rpartial \log\testf\, \Pperp{z}\right).
\end{align*}
\begin{proof}
    We have
\begin{align*}
    \nabla_z\left(|z|^{2+\gamma}\Pperp{z}\nabla_z \log\testf\right) &= \left(\hz\rpartial + \frac{1}{|z|}\hznabla\right) \left(|z|^{1+\gamma} \hznabla \log\testf\right) \\
        &= \left(1+\gamma\right)|z|^\gamma \hz \otimes \hznabla \log\testf + |z|^{1+\gamma}\hz \otimes \hznabla\rpartial \log\testf + |z|^\gamma\hznabla^2 \log\testf
\end{align*}
with
\begin{equation}\label{eq:nabla omega of nabla omega identity}
    \hznabla^2 \log\testf = |z|^2 \left(\Pperp{z} \nabla_z^2 \log\testf \Pperp{z} - \frac{1}{|z|} \rpartial \log\testf \Pperp{z}\right) - \hznabla \log\testf \otimes \hz,
\end{equation}
were we used that
\begin{equation*}
    \hznabla \Pperp{z} = -3 \left(\Pperp{z}\! \otimes \hz\right)_{\mathrm{sym}} + 2 \hz \otimes \Pperp{z}.
\end{equation*}
\end{proof}
\begin{remark}
    We emphasize that $\hznabla^2$, as a matrix, is not (a multiple of) the extrinsic spherical Hessian as shown by \eqref{eq:spherical hessian formula} and \eqref{eq:nabla omega of nabla omega identity}. They differ by the spherical gradient tensorized with the associated unit vector. Nevertheless, in view of \eqref{eq:nabla nabla omega traced pi is spherical Laplacian},
    \begin{equation}
        \Delta_{\S^2} \log \tests = \Pperp{z} : \hznabla^2 \log \testf.
    \end{equation}
    Moreover, a direct computation yields
    \begin{equation}\label{eq:Gamma 2 HS}
        \norm{\hznabla^2 \log \testf}^2_{\mathrm{HS}} = |z|^4 \norm{\left(\Pperp{z}\gradz\right)^2 \log\testf}^2_{\mathrm{HS}} = \Gamma_2\left(\log \tests, \log \tests\right),
    \end{equation}
    where the iterated \emph{carr\'e du champs} $\Gamma_2$ of the spherical Laplacian on $\S^2$ is given by
    \begin{equation}\label{eq:Gamma 2}
        \Gamma_2\left(\log \tests, \log \tests\right) = \norm{\snabla^2 \log \tests}^2_{\mathrm{HS}} + \abs{\snabla \log \tests}^2.
    \end{equation}
\end{remark}
\end{lemma}

\printbibliography
\end{document}